\documentclass[12pt]{article}
\usepackage{verbatim}
\usepackage{amssymb}
\usepackage{amsthm}
\usepackage{amsmath}
\usepackage{mathrsfs}
\usepackage[margin=1in]{geometry}
\usepackage[usenames,dvipsnames] {color}

\definecolor{dark_purple}{rgb}{0.4, 0.0, 0.4}
\definecolor{dark_green}{rgb}{0.0, 0.7, 0.0}

\newcommand\black{\color{black}}
\numberwithin{equation}{section}

\title{Norm Preserving Extensions of Holomorphic Functions Defined on Varieties in $\c^n$}
\author{Jim Agler
\and
\L{}ukasz Kosi\'nski
\thanks{Partially supported by the NCN grant SONATA BIS no. 2017/26/E/ST1/00723}
\and
John E. M\raise.5ex\hbox{c}Carthy
\thanks{Partially supported by National Science Foundation Grant  
DMS 2054199}
}
\date{\today}

\newcommand{\mc}{M\raise.45ex\hbox{c}Carthy}

\def\d{\mathbb{D}}
\def\h{\mathcal{H}}

\def\be{\begin{equation}}
\def\ee{\end{equation}}
\def\norm#1{\| #1 \|}
\def\m{\mathcal{M}}
\def\f{\mathcal{F}}
\def\n{\mathcal{N}}
\def\mn{\mathcal{M}_n}
\def\md{\mathcal{M}^{[d]}}
\def\mone{\mathcal{M}^{[1]}}
\def\mdcom{\mathcal{M}^{[d]}_{\rm com}}

\def\pd{\mathbb{P}_d}
\def\pdcom{{\rm P}_d}
\def\pdij{\mathbb{P}_d^{\, I\times J}}
\def\pdcomij{{\rm P}_d^{\, I\times J}}

\def\c{\mathbb{C}}
\def\d{\mathbb{D}}
\def\t{\mathbb{T}}

\def\h{\mathcal{H}}

\def\l{\mathcal{L}}
\def\f{\mathcal{F}}

\def\set#1#2{\{ #1 \, | \, #2\}}

\def\car{(\Omega,V)}

\def\hol{{\rm Hol}}
\def\hinf{{\rm Hol}^\infty}
\def\ess{\mathscr{S}}
\def\herg{\mathscr{H}}
\def\R{\mathscr{R}}
\def\u{\mathscr{U}}
\def\C{\mathscr{C}}
\def\ip#1#2{\langle#1,#2\rangle}
\def\b{\mathcal{B}}
\def\ball{\rm ball\,}

\DeclareMathOperator{\spn}{span}
\newcommand{\threepartdef}[6]
{
	\left\{
		\begin{array}{ll}
			#1 & \mbox{if } #2 \\ \\
			#3 & \mbox{if } #4 \\ \\
            #5 & \mbox{if } #6
		\end{array}
	\right.
}
\def\O{\Omega}
\def\l{\lambda}
\def\be{\begin{equation}}
\def\ee{\end{equation}}
\def\calg{{\mathcal G}}
\newcommand\calv{{\mathcal V}}

\begin{document}

\bibliographystyle{plain}
\theoremstyle{definition}
\newtheorem{defin}[equation]{Definition}
\newtheorem{lem}[equation]{Lemma}
\newtheorem{prop}[equation]{Proposition}
\newtheorem{thm}[equation]{Theorem}
\newtheorem{claim}[equation]{Claim}
\newtheorem{ques}[equation]{Question}
\newtheorem{remark}[equation]{Remark}
\newtheorem{fact}[equation]{Fact}
\newtheorem{axiom}[equation]{Technical Axiom}
\newtheorem{newaxiom}[equation]{New Technical Axiom}
\newtheorem{cor}[equation]{Corollary}
\newtheorem{exam}[equation]{Example}
\maketitle

\begin{abstract}
If $V$ is an analytic set in a pseudoconvex domain $\Omega$, we show there is always a pseudoconvex domain $G \subseteq \Omega$ that contains $V$ and has the property that every
bounded holomorphic function on $V$ extends to a bounded holomorphic function on $G$ with the same norm. We find such a $G$ for some particular analytic sets.

 When $\Omega$ is an operhedron
we show there is a norm on holomorphic functions on $V$ that can always be preserved by extensions to $\Omega$.
\end{abstract}

%%%%%%%%%%%%%
% Section 1 %
%%%%%%%%%%%%%
\section{Motivation}

Let $\Omega\subset \mathbb C^n$ be a domain of holomorphy and $V$ be a subset. We study when
a bounded, holomorphic function on $V$ extends to a bounded holomorphic function with the same norm defined on the whole domain $\Omega$. If $V$ is a discrete set, this is  an interpolation problem.
In this note  we are interested in the case when $V$ is a fatter set. It is then natural to assume it has an analytic structure.

The first significant result in this direction was obtained in \cite{agmc_vn}, where the simplest non-trivial case $\Omega = \mathbb D^2$ was studied. It was shown using operator theory  that if $V\subset \mathbb D^2$ is algebraic,  then it must be a retract to have the isometric extension property. The function theory argument of P. Thomas \cite{tho03} inspired further research: it turns out that a similar result holds for the $n$-dimensional Euclidean ball and more generally for strictly convex domains and strongly linearly convex domains in $\mathbb C^2$ \cite{kmc19}. In all these cases $V$ was shown to be a retract. Some methods developed there remained true for strongly convex and linearly convex domains in $\mathbb C^n$ with arbitrary $n>1$. In particular, it was shown that  $V$ was always totally geodesic with respect to the Kobayashi metric on $G$ --- consequently, it had no singularities.

The aim of this paper is to show a  result in the opposite direction. Using  a totally different approach, we prove that if $V$ is an analytic set in a pseudoconvex domain $\Omega$, there is always a pseudoconvex domain $G \subseteq \Omega$ that contains $V$ and has the property that every bounded holomorphic function on $V$ extends to a bounded holomorphic function on $G$ with the same norm. We find such a $G$ for some analytic sets. These particular examples show that the situation when $G$ is a convex domain while $V$ has singularities (that are either cusps or obtained as intersections of two curves) can occur. This is intriguing as the crucial role in the strategy of proving extension results has been so far played by the Lempert theory of invariant functions on convex domains.

\section{Introduction}

%\subsection{Overview}
If $\Omega$ is a domain of holomorphy in $\c^n$, then an \emph{analytic set in $\Omega$} is a relatively closed set $V$ in $\Omega$ such that for each point $\lambda\in V$ there exist a neighborhood $U$ of $\lambda$ ($U$ is open in $\c^n$) and holomorphic functions $f_1, f_2,\ldots, f_m$ on $U$ such that
\[
V\cap U =\set{\mu\in U}{f_i(\mu)=0 \text{ for } i=1,2,\ldots,m}.
\]
We say that a function $f:V\to \c$ is \emph{holomorphic on $V$} if for each $\lambda \in V$ there exist an open set $U \subseteq \c^n$ containing $\lambda$ and a holomorphic function $F$ defined on $U$ such that $F(\mu) =f(\mu)$ for all $\mu \in V \cap U$.

The following fact is  one of Henri Cartan's numerous profound contributions to several complex variables.
\begin{thm}
\label{thm11}
(Cartan Extension Theorem) 
If $\O$ is a domain of holomorphy, $V$ is an analytic set in $\O$,
and $f$ is a holomorphic function defined on $V,$ then there exists a holomorphic function $F$ defined on $\Omega$ such that
\[
f = F\, |\, V.
\]
\end{thm}
This theorem, a consequence of Cartan's deep work in the theory of analytic sheaves, was originally presented in \cite{car51}. Accessible proofs can be found in the classic texts \cite{gur} and \cite{hor} (Theorem 7.4.8) and also in the more modern text  \cite{tay} (Corollary 11.5.2).

 This prompts the following definition.
 \begin{defin}\label{car10}
We say that a pair $(\Omega,V)$ is a \emph{Cartan pair} if $\Omega$ is a domain of holomorphy in $\c^n$ for some $n$ and $V$ is an analytic set in $\Omega$.
\end{defin}

In this paper we shall be interested in cases where one can obtain bounds in  Cartan's theorem.
 If $(\Omega,V)$ is a Cartan pair, say that $(\Omega,V)$ has the \emph{bounded extension property} if every bounded holomorphic function on $V$ has a bounded extension to $\Omega$.  In the case when $\Omega=\d^n$, the polydisc, geometric conditions which imply that $(\Omega,V)$ has the bounded extension property have been investigated by Herbert Alexander \cite{ale69}, Walter Rudin \cite{rud69} (Theorem 7.5.1), Edgar Lee Stout \cite{sto75}, Guennadi Henkin and Pierre Polyakov \cite{henpol84}, and Greg Knese \cite{kn08ua}. One lesson to be gleaned from these papers is that the bounded extension property is quite sensitive to the interaction of the boundary of $V$ and $\t^n$, the distinguished boundary of $\d^n$. For example, in \cite{ale69} it is shown that a certain class of varieties in $\d^n$ introduced by Rudin have the bounded extension property if and only if $\partial V \cap \t^d = \varnothing$ and in \cite{kn08ua} it is shown that if $V$ is an algebraic distinguished variety in $\d^2$ (i.e., $V$ is an algebraic set in $\c^2$ and $\partial (V \cap \d^2) \subseteq \t^2$) with no singularities on $\t^2$, then $(\d^2,V \cap \d^2)$ has the bounded extension property. Other work on the bounded extension property,
 when $\Omega$ is assumed to be strictly pseudoconvex (but its  boundary is not assumed to be smooth) and $V$ is  assumed to have the form $V=D\cap \Omega$, where $D$ is a relatively closed complex submanifold of a neighborhood of $\Omega^-$, can be found in \cite{henlei84} and \cite{aac99}.
 
 A stronger notion was studied in \cite{agmc_vn}, namely when an extension can always be found with the same norm. This will be the focus of the current paper.
\begin{defin}\label{car20}
 If $(\Omega,V)$ is a Cartan pair, then we say that $(\Omega,V)$ has the \emph{norm preserving extension property} (or simply, $(\Omega,V)$ is an \emph{np pair}) if for every bounded holomorphic function $f$ on $V$ there exists a holomorphic function $F$ on $\Omega$ such that $F(\mu) =f(\mu)$ for all $\mu \in V$ and
 \[
 \sup_{\mu \in \Omega} |F(\mu)| = \sup_{\mu \in V} |f(\mu)|.
 \]
 \end{defin}
 We organize our results in this paper around the following three fundamental problems associated with the study of np pairs. \\ \\
{\bf Problem A.} Given a domain of holomorphy $\Omega$, find all analytic sets $V$ in $\Omega$ such that $(\Omega,V)$ is an np pair.\\ \\
{\bf Problem B.} Given a Cartan pair $(\Omega,V)$, find all domains of holomorphy $G$ such that $(G,V)$ is an np pair.\\ \\
{\bf Problem C.} Given a Cartan pair $(\Omega,V)$, how can one concretely construct a domain of holomorphy $G$ such that $(G,V)$ is an np pair?
 
\section{Results}

\subsection{Problem A}

It is obvious that if $V$ is a holomorphic retract of $\O$ (which means there is a holomorphic 
function $r: \O \to V$ which equals the identity on $V$) then $f \mapsto f \circ r$ is a norm preserving homomorphism that gives extensions.
On the bidisk this turns out to be the only way that a reasonably nice set can have the norm-preserving extension property.
We say that $V$ is relatively polynomially convex in a domain $\O$ if $\overline{V}$ is polynomially convex and $\overline{V} \cap \O = V$.

\begin{thm}
(Theorem 1.20 from \cite{agmc_vn}) Let $V$ be a nonempty relatively polynomially convex subset of $\d^2$. Every polynomial has a norm preserving extension to $\d^2$ if and only if $V$ is a holomorphic retract of $\d^2$.
\end{thm} 
Non-trivial retracts of $\d^2$ are of the form $\{(z, \phi(z)) : z \in \d\}$
or $\{( \phi(z),z) : z \in \d\}$ for some holomorphic $\phi: \d \to \d$ \cite{hs81},
so in particular if $V$ is a  relatively polynomially convex subset of $\d^2$
and $(\d^2, V)$ is norm preserving, then $V$ is a manifold.

Both these results hold more generally. In \cite{kmc19}, it was shown that if $(\O,V)$ is an np pair and $V$ is  relatively polynomially convex, then $V$ must be a retract if any of the following hold:
\begin{itemize}
\item
$\O$ is a ball in any dimension;
\item
$\O$ is strictly convex in $\c^2$;
\item
 $\O$ is strongly linearly convex in $\c^2$ with $C^3$ boundary.
\end{itemize}
It was also proved that $V$ must at least be a  complex submanifold if either
$\O$ is the tridisk \cite{kmc20} or is strongly linearly convex  with $C^3$ boundary in any dimension.

It came as a surprise, then, when the first author, together with Z. Lykova and
N. Young, solved Problem A for the case that $\O$ is the symmetrized bidisk and $V$
an algebraic subset \cite{aly19}; see also \cite{bhsa18}. It was shown that some special $V$'s arise that are not
only not retracts, but are not even submanifolds. This result spurred the current investigation.

\subsection{Problem B}

In Section \ref{sec3} we show that the geometry of $V$ has very little to do with whether
there is an np pair $(G,V)$. 

{\bf Theorem \ref{exist.thm.20}.}
If $(\Omega,V)$ is a Cartan pair, then there exists $G \subseteq \Omega$ such that $(G,V)$ is a norm preserving Cartan pair if and only if $V$ is connected.

\vskip 5pt
In Section \ref{secd} we study the situation for the set $T = \d \times \{0 \} \cup \{ 0 \} \times \d$
in $\c^2$, consisting of two crossed disks.
We prove:

{\bf Theorem \ref{thm44}.}
Let $G$ be a balanced domain of holomorphy in $\c^2$ with $T\subseteq G$. $(G,T)$ is norm preserving if and only if $G \subseteq  \Delta = \set{\lambda}{\ |\lambda_1|+|\lambda_2|<1}$.

\vskip 5pt

In Section \ref{seclinear}, we show that though $(\Delta, T)$ is np, there is no linear isometric extension.

Without the hypothesis of balanced, the description of all np pairs $(G,T)$ is more complicated.

{\bf  Theorem \ref{thm45}.}
Let $G$ be a domain in $\c^2$. Then
$(G,T)$ is an np pair if and only if $T$ is a relatively closed subset of $G$ and there exist a pseudoconvex set $U$ in $\c^2$ and a  function $\tau \mapsto C_\tau$ from $\t^2$ into $\hol(U)$ such that
\begin{equation}
\label{eq46}
G= \bigcap_{\tau \in \t^2} \set{\lambda \in U\ }{\ \ |\tau \cdot \lambda +\lambda_1\lambda_2 C_\tau(\lambda)|<1}.
\end{equation}

\subsection{Problem C}
In Sections \ref{sec7} and \ref{sec8}, we study the set 
\be
\label{eq23}
\calv \ =\ \set{z \in \d^3}{z_3^2 =z_1z_2} .
\ee
We show how one is naturally led to consider the domain
\be
\label{eq24}
\calg \ = \
  \set{z \in \d^3}{|z_1 z_2 - z_3^2| < (1-|z_3|^2) + \sqrt{1-|z_1|^2}\sqrt{1-|z_2|^2}} .
\ee

{\bf Theorems \ref{thm814} and \ref{thm828}.}
Let $\calv$ be given by \eqref{eq23}, and $\calg$ be  defined by \eqref{eq24}. Then 
$\calg$ is convex, and 
 $(\calg,\calv)$ is an np pair. Moreover, if $G$ is any balanced set in $\c^3$, then
$(G, \calv)$ is an np pair if and only if $G \subseteq \calg$.

\vskip 5pt

Our construction depends on writing down a $2$-to-$1$ branched cover of $\calv$ by $\d^2$,
and developing a model for even bounded holomorphic functions on $\d^2$ in Section \ref{sec6}.

\subsection{Noncommutative np extensions}

In Section \ref{sec9} we study np extensions in the free setting, and 
prove in Theorems~\ref{free.thm.40} and \ref{free.thm.50} that in some generality,
norm preserving extensions always exist, but the norm that is preserved is not
the supremum of $|f|$ evaluated on points in $\c^d$, but rather the supremum of $\| f \|$ evaluated
on $d$-tuples of matrices. 

This allows us to prove a version of Cartan's Theorem \ref{thm11}.
We restrict $\O$ to be an operhedron (a slight generalization of a poynomial polyhedron in $\c^n$) and
$V$ to be an algebraic subset, but we gain norm estimates: we show we can always extend a function
$f$ on $V$ to a function $F$ on $\O$ with the {\em same norm}, where again the norms are
given by taking the supremum of $\| f(x) \|$  (resp. $\| F (y) \|$) where $x$ ranges over 
certain $d$-tuples of matrices associated with $V$, and $y$ ranges over $n$-tuples associated with $\O$.
See Theorem~\ref{free.thm.60} for an exact statement.

Although the proof of Theorem~\ref{free.thm.60} uses non-commutative function theory, the statement
is just a theorem about extending holomorphic functions from analytic sets to domains in $\c^n$.

\section{Existence of Norm Preserving Pairs}
\label{sec3}
Rudin \cite{rud69} was the first to exploit the following theorem, essentially a consequence of the Michael Selection Theorem \cite{mic56}, to study extensions of holomorphic functions defined on subvarieties on polydiscs.
\begin{thm}\label{exist.thm.10}
If $X$ and $Y$ are Frechet spaces and $L:X \to Y$ is a continuous linear surjection, then there exists a continuous function $S:Y \to X$ satisfying $L(S(y))=y$ for all $y\in Y$.
\end{thm}
\begin{proof}
The theorem is a special case of the far more general Michael Selection Theorem \cite{mic56}. Alternately, for a brief, direct proof consult the proof of Theorem 7.2.3 in \cite{rud69}.
\end{proof}
If $\Omega$ is a domain we let $\hol(\Omega)$ denote the collection of holomorphic functions on $\Omega$, let $\hinf(\Omega)$ denote the collection of bounded holomorphic functions on $\Omega$, and let $\ess(\Omega)$ denote the Schur class of $\Omega$, i.e., the subset of $\hinf(\Omega)$ defined by
\[
\ess(\Omega)=\set{\phi \in \hol(\Omega)}{\sup_{\lambda \in \Omega} |\phi(\lambda)|\le 1}.
\]
We let $\herg(\O)$ denote the Herglotz class of $\O$, namely
\[
\herg(\Omega)=\set{\phi \in \hol(\Omega)} { \forall \lambda \in \Omega, \ \Re \phi (\lambda) \geq 0 }.
\]

In similar fashion, if $(\Omega,V)$ is a Cartan pair, then we let $\hol(V)$ denote the collection of holomorphic functions on $V$, let $\hinf(V)$ denote the collection of bounded holomorphic functions on $V$, and let $\ess(V)$  and $\herg(V)$ denote the Schur  and Herglotz classes of $V$, i.e., those subsets of  $\hol(V)$ \black defined by
\begin{eqnarray*}
\ess(V) &\ =\  & \set{\phi \in \hol(V)}{\sup_{\lambda \in V} |\phi(\lambda)|\le 1}\\
\herg(V) &= & \set{\phi \in \hol(\Omega)} { \forall \lambda \in V, \ \Re \phi (\lambda) \geq 0 }.
\end{eqnarray*}

It is well known and elementary that $\hol(\Omega)$ is a Frechet space when equipped with the topology of uniform convergence on compact subsets of $\Omega$. That $\hol(V)$ is a Frechet space as well, when so equipped, is a very deep result implicit in the original seminars of Cartan \cite{car51} (see Theorem~7.4.9 \cite{hor}). As consequence, we are able to apply Theorem \ref{exist.thm.10} to prove the following highly useful technical result.
\begin{lem}\label{exist.lem.10}
If $(\Omega,V)$ is a Cartan pair, then their exists a continuous function $S:\hol(V) \to \hol(\Omega)$ such that $S(f)|V =f$ for all $f \in \hol(V)$
\end{lem}
\begin{proof}
By the Cartan Extension Theorem (Theorem 7.4.8 in \cite{hor}), the map $L$ defined by $L(F)=F|V$ is a continuous linear surjection from $\hol(\Omega)$ onto $\hol(V)$. Therefore, as $\hol(\Omega)$ and $\hol(V)$ are Frechet spaces, the desired continuous function $S$ exists by Theorem \ref{exist.thm.10}.
\end{proof}
Another useful lemma is the following. Recall that a TVS is said to be a Montel space if it is barrelled and the Heine-Borel Theorem is true, i.e., a set is compact if and only if it is closed and bounded.
\begin{lem}\label{exist.lem.20}
If $(\Omega,V)$ is a Cartan pair, then $\hol(V)$ is a Montel space.
\end{lem}
\begin{proof}
By \cite[Thm. 2.4.2]{tay}, $\hol(\O)$ is a Montel space. By Cartan's theorem, $\hol(V)$ is the quotient of $\hol(\O)$ by the 
closed subspace of functions vanishing on $V$, and so is a separated quotient.
By \cite[Prop. 11.1.1]{tay} a 
 separated quotient of a Montel space is Montel.
\end{proof}
For a final lemma we have the following result.
\begin{lem}\label{exist.lem.30}
Let $(\Omega,V)$ be a Cartan pair, and assume $V$ is connected.  Fix a point $\lambda_0 \in V$, and let
\[
\ess_{\lambda_0} (V) =\set{f \in \ess(V)}{f(\lambda_0)=0}.
\]
Then $(\Omega,V)$ is norm preserving if and only if each $f \in \ess_{\lambda_0}(V)$ has an extension to an element $F \in \ess(\Omega)$.
\end{lem}
\begin{proof}
Let $f \in \ess(V)$. If $|f(\lambda_0)|=1$, then by the Maximum Principle (\cite{gur} Theorem 16, Chapter III) $f$ is constant on $V$ and trivially has a norm preserving extension to $\Omega$. Otherwise, $|f(\lambda_0)|<1$ and there exists a M\"obius transformation $m$ of $\d$ such that $(m \circ f) (\lambda_0)=0$. As $h=m \circ f \in \ess(V)$, the assumption of the lemma implies that there exists $H \in \ess(\Omega)$ such that $H|V=h$. But then if we define $F=m^{-1} \circ H$, $F \in \ess(\Omega)$ and $F|V =f$.
\end{proof}
\begin{thm}\label{exist.thm.20}
If $(\Omega,V)$ is a Cartan pair, then there exists $G \subseteq \Omega$ such that $(G,V)$ is a norm preserving Cartan pair if and only if $V$ is connected.
\end{thm}
\begin{proof}
If $V$ is not connected, there is a non-trivial characteristic function in $\hinf(V)$. 
By the maximum principle, any extension of this to a pseudoconvex domain $G$ must have a larger
norm.

Conversely, assume $V$ is connected.
Fix $\lambda_0 \in V$ and let $S$ be as in Lemma \ref{exist.lem.10}. For each $f\in \ess_{\lambda_0}(V)$ define $\Omega_f$ in $\Omega$ by
\[
\Omega_f=\set{\lambda \in \Omega}{|S(f)(\lambda)|<1}
\]
and let $G$ be the connected component of 
\[
\Big(\bigcap_{f\in \ess_{\lambda_0}(V)}\Omega_f\Big)^\circ.
\]
containing $V$.

We first observe that, as every connected component of each  $\Omega_f$ is a domain of holomorphy
\cite[Prop. 4.1.7]{jj20},  it follows that $G$ is a domain of holomorphy  \cite[Cor. II.3.19]{ran}. 
Hence, $(G,V)$ will be a Cartan pair if
$V \subseteq G$. Assume to the contrary that $\mu \in V \setminus G$. Then for each positive integer $n$, there exist $\mu_n \in B(\mu,1/n)\cap \Omega$ and $f_n \in \ess_{\lambda_0}(V)$ with $|S(f_n)(\mu_n)| \ge 1$. As Lemma \ref{exist.lem.20} implies that $\ess_{\lambda_0}(V)$ is a compact subset of $\hol(V)$, there exists a subsequence $\{f_{n_k}\}$ and $f\in \ess_{\lambda_0}(V)$ such that $f_{n_k} \to f$ in $\hol(V)$ as $k \to \infty$. It follows by the continuity of $S$, that $|f(\mu)|\ge 1$. But as $f \in \ess_{\lambda_0}$, this contradicts the Maximum Principle.

Now observe that by construction,
\[
f\in \ess_{\lambda_0}(V) \implies \sup_{\lambda \in G}|Sf(\lambda)|\le 1.
\]
Hence, if $f\in \ess_{\lambda_0}(V)$ and
\[
\sup_{\lambda \in V}|f(\lambda)|=\rho \not=0,
\]
then, as $\rho^{-1}f \in \ess_{\lambda_0}(V)$,
\[
\sup_{\lambda \in G}|S(\rho^{-1}f)(\lambda)|\le 1,
\]
which implies that $\rho S(\rho^{-1}f)$ is a norm preserving extension of $f$ to $G$. Consequently, Lemma \ref{exist.lem.30} implies that $(G,V)$ is norm preserving.
\end{proof}

\begin{remark} Some authors, such as H\"ormander \cite{hor}, do not require that a domain of holomorphy be connected; in this case the connectivity of $V$ can then also be dropped.
Indeed, we just need to find disjoint pseudoconvex subdomains of $\O$ each of which contains at most one
component of $V$. To do this, let $\{ V_j : j  = 1, \dots \}$ be an enumeration of the at most countably many components of $V$. Let $f$ be the function that is $0$ on $V_1$ and $1$ on
each other $V_j$.
Since $(\O,V)$ is a Cartan pair, there is a holomorphic $F$ on $\O$ extending $f$.
Let $\O_1 = \{ \l \in \O : |F(\l) | < 1/2 \}$. Let $D = \{ \l \in \O : |F(\l) - 1 | < 1/2 \}$.
Apply Theorem~\ref{exist.thm.20} to $(\O_1, V_1)$ to get a norm preserving pair $(G_1, V_1)$
with $G_1 \subseteq \O_1$. Now replace $(\O,V)$ with $(D, \cup_{j \geq 2} V_j)$ and continue by induction.
\end{remark}

\section{Problems B and C for Two Crossed Discs}
\label{secd}

In this section we let $D_1$ and $D_2$ be the two sets in $\c^2$ defined by
\[
D_1=\set{(z,0)}{z \in \d}\qquad \text{ and }\qquad D_2 =\set{(0,z)}{z\in\d}
\]
and define an analytic set $T$ in $\c^2$ by
\[
T = D_1 \cup D_2.
\]
It is easy to construct domains of holomorphy $\Omega$ in $\c^2$ such that $(\Omega,T)$ is a Cartan pair. For example, as
\[
T=\set{\lambda \in \d^2}{\lambda_1\lambda_2=0},
\]
$(\d^2,T)$ is a Cartan pair. In this section we shall identify conditions which will guarantee that $(\Omega,T)$ is a norm preserving Cartan pair.

\begin{lem}
\label{lem41}
A function $f$ is holomorphic on $T$ if and only if there exist holomorphic functions $f_1$ and $f_2$ defined on $\d$ such that $f_1(0)=f_2(0)$ and $f$ is given by the formula
\begin{equation}
\label{eq41}
  f(\lambda) =
\left\{
	\begin{array}{ll}
		f_1(z)  & \mbox{if } \lambda=(z,0)\in D_1 \\ \\
		f_2(z) & \mbox{if } \lambda=(0,z)\in D_2.
	\end{array}
\right.
\end{equation}
\end{lem}
The proof of Lemma~\ref{lem41} is 
straightforward
 and follows from the fact that a function given by formula \eqref{eq41} can be extended to a neighborhood of $T$ (a possible extension to $\d^2$ is given by the formula $F(z,w) = f(z,0) + f(0,w) - f(0,0)$).

Let $\herg_{\lambda_0}(V)$ denote the holomorphic functions of positive real part that map
$\lambda_0$ to $1$.
For each $\tau \in \t^2$, we may define $b_\tau \in \ess_{(0,0)}(T)$ and $h_\tau \in \herg_{(0,0)} (T)$ by the formulas
\[
  b_\tau(\lambda) =
\left\{
	\begin{array}{ll}
		\tau_1 z  & \mbox{if } \lambda=(z,0)\in D_1 \\ \\
		\tau_2 z & \mbox{if } \lambda=(0,z)\in D_2
	\end{array}
\right.
\]
and
\[
  h_\tau(\lambda) =
\left\{
	\begin{array}{ll}
		\frac{1+\tau_1 z}{1-\tau_1 z}  & \mbox{if } \lambda=(z,0)\in D_1 \\ \\
		\frac{1+\tau_2 z}{1-\tau_2 z} & \mbox{if } \lambda=(0,z)\in D_2.
	\end{array}
\right.
\]
Since the map
\[
C: z \mapsto \frac{1 + z}{1-z}
\]
is a conformal map from the unit disk onto the right-half plane, it is immediate
that one can extend the  function $f$ from $\ess(T)$ to $\ess(G)$ if and only if
one can extend $C \circ f$ from $\herg(T)$ to $\herg(G)$.

\begin{lem}\label{lem42}
If $(G,T)$ is a Cartan pair, then $(G,T)$ is norm preserving if and only if for each $\tau \in \t^2$, there exists $B_\tau \in \ess(G)$ such that $B_\tau|T=b_\tau$.
\end{lem}
\begin{proof}
First note that necessity is obvious. To show sufficiency assume that for each $\tau \in \t^2$, there exists $B_\tau \in \ess(G)$ such that $B_\tau|T=b_\tau$, or equivalently
that
\[
\forall_{\tau\in \t^2}\ \ \exists_{H_\tau \in \herg(G)}\ \
H_\tau |T = h_\tau.
\]
To prove that $(G,T)$ is a norm preserving pair, by Lemma \ref{exist.lem.30} it suffices to show that if $h \in \herg_{(0,0)}(T)$ then there exists $H\in \herg(G)$ such that $H|T=h$.

Assume that $h \in \herg_{(0,0)}(T)$. By Lemma \ref{lem41}, there exist $h_1,h_2 \in \herg_0(\d)$ such that
\[
  h(\lambda) =
\left\{
	\begin{array}{ll}
		h_1(z)  & \mbox{if } \lambda=(z,0)\in D_1 \\ \\
		h_2(z) & \mbox{if } \lambda=(0,z)\in D_2.
	\end{array}
\right.
\]
By the Herglotz Representation Theorem, there exist probability measures $\mu_1,\mu_2$ on $\t$ such that
\[
h_1(z) = \int \frac{1+\tau z}{1-\tau z}\ d\mu_1(\tau)\qquad \text{ and }\qquad
h_2(z) = \int \frac{1+\tau z}{1-\tau z}\ d\mu_2(\tau).
\]
Let $\mu=\mu_1\times\mu_2$ (so that $\mu$ is a probability measure on $\t^2$).
First, suppose that the map $\tau \mapsto H_\tau$ is measurable.
Then we  define $H$ by the formula
\[
H(\lambda)=\int_{\tau \in \t^2} H_\tau(\lambda)\ d\mu(\tau);
\]
Clearly, as $\mu$ is a probability measure and $H_\tau\in \herg(G)$ for all $\tau \in \t^2$, $H\in \herg(G)$.
If $\lambda =(z,0) \in D_1$,
\begin{align*}
H(\lambda) &= \int_{\tau \in \t^2} H_\tau(\lambda)\ d\mu(\tau)\\ \\
&= \int_{\tau \in \t^2} h_\tau(\lambda)\ d\mu(\tau)\\ \\
&=\int_{\tau_2\in \t}\Big(\int_{\tau_1\in \t} \frac{1+\tau_1 z}{1-\tau_1 z}\ d\mu_1(\tau_1)\Big)\ d\mu_2(\tau_2)\\ \\
&=h_1(z)\\ \\
&=h(\lambda).
\end{align*}
Likewise, if $\lambda =(0,z) \in D_2$, $H(\lambda)= h(\lambda)$.
Consequently, $H\in \herg(G)$ and $H|T=h$, as was to be proved.

Now, let us drop the measurability hypothesis; we shall use an approximation argument to get round
the difficulty.
Let $K_n$ be a compact exhaustion of $G$. It thus follows
 (use e.g.  a Montel type argument) that there is a constant  $\delta_{n} > 0$ so that 
\[ \| B_\tau \|_{K_n} \  \leq \ 1 - \frac{1}{\delta_{n}} \quad \forall \ \tau \in \t^2 .
\]
\black 
Let
\[
M_{n} \ := \ 
\sup \{  \| \lambda \|_{\infty} : \lambda \in K_n \} \  < \infty .
\]
 Choose a finite set of points $F_{n}$ in $\t^2$ so that every point in
$\t^2$ is less than $\delta_{n}/M_{n}$ away from some point of $F_{n}$ in the $\ell^1$ norm. Let 
$\alpha: \t^2 \to F_{n}$ be a Borel map such that $\| \alpha(\tau) - \tau \|_{\ell^1} < \delta_{n}/M_{n}$ for all $\tau$.
Write
\[
B_\tau(\lambda) \ = \ \lambda_1 \tau_1 + \lambda_2 \tau_2  + 
\lambda_1 \lambda_2 \Gamma_\tau(\lambda).
\]
Define 
\[
B^{(n)}_\tau(\lambda) \ = \ \lambda_1 \tau_1 + \lambda_2 \tau_2 
+ 
\lambda_1 \lambda_2 \Gamma_{\alpha(\tau)} (\lambda).
\]
Then the map $\tau \mapsto B^{(n)}_\tau$ is Borel, and as
\[
| B^{(n)}_\tau (\l) - B_{\alpha(\tau)}(\lambda) | 
\ \leq \ M_{n} \| \tau - \alpha(\tau) \|_{\ell^1} \| \lambda \|_{\ell^\infty},
\]
we have that $ |B^{(n)}_\tau | < 1$ on $K_n$.
So if we define
 $H^{(n)}_\tau = C \circ B^{(n)}_\tau$,
 then $H^{(n)}_\tau$ is a holomorphic map from $K_n$ to the right-half plane.
Let
\[
H^{(n)}(\lambda)=\int_{\tau \in \t^2} H^{(n)}_\tau(\lambda)\ d\mu(\tau).
\]
The sequence $H^{(n)}$ will have a subsequence that converges uniformly on
compact subsets of $G$ to some function $H$ in $\herg(G)$; this will be our
desired extension of $h$.
\end{proof}
For $v$ a unit vector in $\c^2$ and $G$ a neighborhood of the origin in $\c^2$ let
\[
R_v = \sup\  \set{r}{\forall_{z\in \c}\ \  |z|<r \implies zv \in G}.
\]
\begin{lem}\label{lem43}
Assume that $G$ a neighborhood of the origin in $\c^2$ and $v$ is a unit vector in $\c^2$. If $(G,T)$ is np, then $R_v \le \frac{1}{|v_1|+|v_2|}$.
\end{lem}
\begin{proof}
By Lemma \ref{lem42}, for each $\tau\in \t$, there exists $B_\tau \in \ess(G)$ such that $B_\tau|T=b_\tau$. Evidently, we have that
\[
\frac{\partial B_\tau}{\partial \lambda_1}(0,0)=\tau_1\qquad \text{ and }\qquad
\frac{\partial B_\tau}{\partial \lambda_2}(0,0)=\tau_2.
\]
Hence, by the chain rule, if we define an analytic function $f$ on $\d$ by the formula
\[
f(z)=B_\tau(zR_v v),
\]
\[
f'(0)=(\tau_1v_1 +\tau_2v_2)R_v.
\]
Hence, by Schwarz's Lemma,
\[
|\tau_1v_1 +\tau_2v_2| R_v \le 1.
\]
As $\tau \in \t^2$ is arbitrary, the result follows.
\end{proof}
\begin{thm}\label{thm44}
Let $G$ be a balanced domain of holomorphy in $\c^2$ with $T\subseteq G$. $(G,T)$ is norm preserving if and only if $G \subseteq \set{\lambda}{\ |\lambda_1|+|\lambda_2|<1}$
\end{thm}
\begin{proof}
Lemma \ref{lem43} implies necessity and Lemma \ref{lem42} implies sufficiency.
\end{proof}
In the unbalanced case things are not so clean. However, the following theorem obtains.
\begin{thm} \label{thm45}
Let $G$ be a domain in $\c^2$. Then
$(G,T)$ is an np pair if and only if $T$ is a relatively closed subset of $G$ and there exist a pseudoconvex set $U$ in $\c^2$ and a function $\tau \mapsto C_\tau$ from $\t^2$ into $\hol(U)$ such that
\begin{equation}
\label{eq46}
G= \bigcap_{\tau \in \t^2} \set{\lambda \in U\ }{\ \ |\tau \cdot \lambda +\lambda_1\lambda_2 C_\tau(\lambda)|<1}\Big.
\end{equation}
\end{thm}

\begin{proof}
Suppose $G$ has the form \eqref{eq46}. Since $U$ is pseudoconvex, the  pre-image
of any pseudoconvex set under a holomorphic map is pseudoconvex \cite[Prop. 4.1.7]{jj20}, so each of the sets indexed by $\tau$ on the right-hand side of \eqref{eq46} is
pseudoconvex. Therefore the interior of their intersection is
\cite[Cor. II.3.19]{ran}, and since $G$ is assumed open, this means that
 $G$ is a domain of holomorphy, and hence $(G,T)$ is a Cartan pair.
Every $b_\tau$ has an extension to the function
\begin{equation}
\label{eq47}
B_\tau(\lambda) \ = \ 
\tau \cdot \lambda +\lambda_1\lambda_2 C_\tau(\lambda) 
\end{equation}
in $\ess(G)$, so by Lemma \ref{lem42} we conclude that $(G,T)$ is np.

Conversely, suppose $(G,T)$ is np. By Lemma \ref{lem42}, there is a map
$\tau \mapsto B_\tau$ where each $B_\tau$ is a norm preserving extension of $b_\tau$, and therefore of the form \eqref{eq47}. Choosing $U = G$ we get \eqref{eq46}.
\end{proof}

Let $\Delta = \set{\lambda}{\ |\lambda_1|+|\lambda_2|<1}$.
Using Theorem \ref{thm45} one can construct simple $G$'s that neither contain or are contained in $\Delta$ and such that $(G,T)$ is norm preserving---see e.g. Example \ref{exam49}. 
However, one cannot construct $G$'s that properly contain $\Delta$ such that $(G,T)$ is norm preserving.
\begin{thm}
If $G$ properly contains $\Delta$, then $(G,T)$ is not norm preserving.
\end{thm}
\begin{proof}
There is some point $P$ on $\partial \Delta$ that is not in $\partial G$.
Therefore the intersection of the  complex line  through $0$ and $P$ with $G$ is a set $U$ that properly
contains the unit disk in that plane.

Define
\[
  f(\lambda) =
\left\{
	\begin{array}{ll}
		 \frac{\overline{P_1}}{|P_1|} z  & \mbox{if } \lambda=(z,0)\in D_1 \\ \\
		 \frac{\overline{P_2}}{|P_2|} z & \mbox{if } \lambda=(0,z)\in D_2,
	\end{array}
\right.
\]
where we interpret $0/0$ as $0$. Then by Schwarz's lemma, $f$ cannot have a norm-preserving extension to $U$, and therefore not to $G$ either.
\end{proof}

\begin{exam}
\label{exam49}
Define
\[
C_\tau(\lambda) \ = \  (\tau \cdot \lambda) ,
\]
and define $G$ by \eqref{eq46}, where $U = \d^2$.
Then 
\[
G \ = \ \{ \l \in \d^2 | \ ( | \l_1 | + | \l_2 |)  | 1 +  \l_1 \l_2 | < 1 \} ,
\]
so $G$ is open and connected, and hence $(G,T)$ is an np pair.
However $G$ is not contained in $\Delta$, since it contains
the points $(r,-r)$ for $0 \leq r < 1$.
\end{exam}

%\green
%Observation. If $V = \d \times \{ 0 \}$ in $\c^2$, then any $\Omega = \{ | z_1 + z_2 h(z_1, z_2) | < 1 \}$ has $V$ as a retract, and so $(\O,V)$
%is an np pair. I am pretty sure this means there is no largest balanced $\Omega$.
%\black

\section{Linear isometric extension operators}
\label{seclinear}

In \cite[Sec. 7.5]{rud69} Rudin asks if $(\O,V)$ is a Cartan pair with the
bounded extension property, 
is there a bounded linear extension operator? In this section we show that the answer is no in the case of norm preserving extensions.

We shall say that a Cartan pair $\car$ has the {\em linear norm preserving extension property}
if there is a linear isometry $E: \hinf(V) \to \hinf(\O)$ such that $(E f)(\lambda) = f(\lambda)$
for every $\lambda$ in $V$.

We show in Theorem \ref{thm51} that $(\Delta,T)$ does not have the linear norm preserving extension property.
There are Cartan pairs $(G,T)$ that do have the linear norm preserving
extension property.

Let 
\[
m_a(z) \ =\ \frac{a - z}{1 - \bar a z} 
\]
be the M\"obius map swapping $a$ and $0$.
Whenever $f_1, f_2$ are in $\hol(\d)$ and 
$f_1(0) = f_2(0)$, let
$\Psi_{f_1,f_2}$ be the function in $\hol(T)$ given by the right-hand side of \eqref{eq41}.
\begin{thm}
\label{thm51}
The pair $(\Delta, T)$ does not have the linear
norm preserving extension property.
\end{thm}
\begin{proof}
Suppose there were an isometric linear extension operator $E$.
Let the extensions of $\Psi_{z,0}$ and $\Psi_{0,z}$ be given by
\begin{eqnarray*}
E(\Psi_{z,0}) (\lambda) &\ = \ & \lambda_1 + \lambda_1 \l_2 f_1(\l) \\
E(\Psi_{0,z}) (\lambda) &\ = \ & \lambda_2 + \lambda_1 \l_2 g_1(\l)
\end{eqnarray*}
where $f_1,g_1$ are  holomorphic functions on $\Delta$.
For unimodular $\tau_1$ and $\tau_2$ we get, by linearity of $E$, that 
$$
E\circ \Psi_{\tau_1 z, \tau_2 z} : \l \mapsto  \tau_1 \l_1 + \tau_2 \l_2 + \l_1 \l_2 (\tau_1 f_1(\l) + \tau_2 g_1(\l))$$
 maps $\Delta$ into the unit disc. For $t\in [0,1]$
 $$\psi_t:\zeta \mapsto E(\Psi_{\tau_1 z_1, \tau_2 z_2})(\bar\tau_1 t \zeta, \bar \tau_2 (1-t) \zeta)$$ is thus a well defined holomorphic
  selfmap of the unit disc such that $\psi_t(0)=0$ and $\psi'_t(0)=1$. Therefore, $\psi_t(\zeta) = \zeta$ for every $\zeta \in \d$ and $t\in [0,1]$.
So $\tau_1f_1 + \tau_2 g_1 \equiv 0$ on the
   three real dimensional set $\{(\bar\tau_1 t \zeta, \bar \tau_2 (1-t) \zeta):\ \zeta\in \d,\ t\in [0,1]\}.$ Consequently $\tau_1 f_1 + \tau_2 g_1 = 0$, and as this is true for all choices of
   $\tau$, we get  $f_1=g_1=0$. 
By the maximum principle, $E(1) = 1$, so $E$ is the identity on all affine polynomials.

Now let us consider  $\Psi_{m_a(z), m_a(\omega z)}$, where $\omega \in \t$, and $a \in \d \setminus \{ 0 \}$.
 Observe that 
 $$
 E(\Psi_{m_a(z), m_a(\omega z)})(\l) = a - (1-|a|^2) (\l_1 + \omega \l_2) + \l_1 \l_2 h(\l),$$
 %because the error term is $O(||\l||^2)$ on $T$.
 for some $h \in \hol(\Delta)$.

  It follows from the Schwarz lemma that the function
  \[
  \zeta \mapsto E(\Psi_{m_a(z), m_a(\omega z)})(\zeta/2, \bar \omega \zeta/2)
  \]
   is equal to the M\"obius function $m_a(\zeta)$.
It follows that the mappings $E(\Psi_{m_a(z), m_a(\omega z)})(\l)$ and $m_a(\l_1 + \omega \l_2)$ coincide on $\l_1=0$, $\l_2=0$ and $\l_2 = \bar \omega \l_1$; in particular
\begin{equation}\label{eq: ma1}
E(\Psi_{m_a(z), m_a(\omega z)})(\l) = m_a(\l_1+ \omega \l_2) + \l_1 \l_2 (\bar \omega \l_1 - \l_2) F_{\omega}(\l_1,\l_2).\end{equation}
 
To get a contradiction we shall compute $E(\Psi_{m_a(z), m_a(\omega z)})$ using a Taylor expansion of $m_a$. Let $f_n,g_n$ be functions in $\hol(\Delta)$  such that
\begin{eqnarray*}
E(\Psi_{z^n,0}) (\l) & \ = & \l_1^n + \l_1 \l_2 f_n(z) \\
E(\Psi_{0,z^n}) (\l) &= &  \l_2^n + \l_1 \l_2g_n(\l) .
\end{eqnarray*}

Write $m_a(z) = \sum_n a_n z^n$, where $a_0 = a$ and $a_n = - (1-|a|^2) \bar a^{n-1}$ for $n \geq 1$.
%Note that $\sum|a_n|= 1 + 2|a|$. 
Then 
\begin{eqnarray*}
 E(\Psi_{m_a(z), m_a(\omega z)})(\l)  &\ =\ & E(\Psi_{m_a, 0}) (\l) + E(\Psi_{0,m_a( \omega z)})(\l) - a \\
&=&
a - (1-|a|^2) (\l_1 + \omega \l_2) + \sum_{n\geq 2} a_n ( \l_1^n + \omega^n \l_2^n) + \\
&&\quad \l_1 \l_2
\sum_{n\geq 2} a_n ( f_n(\l) + \omega^n g_n(\l))\\
&=&
 m_a(\l_1 + \omega \l_2) + 2 \omega \bar a (1-|a|^2) \l_1 \l_2 +\\
&&\quad   \l_1 \l_2\sum_{n\geq 2} a_n f_n(\l) + \l_1 \l_2\sum_{n\geq 2} \omega^n a_n g_n (\l) +  O(||\l||^3).
\end{eqnarray*}
Comparing the last equation with \eqref{eq: ma1}, we see that the coefficients of $\l_1 \l_2$
must agree, which yields
\begin{equation}
0 \ = \ 2 \omega \bar a (1-|a|^2) + \sum_{n\geq 2} a_n (f_n(0) + \omega^n g_n(0)) .
\label{eq53}
\end{equation}
As \eqref{eq53} holds for all $\omega\in \t$, we get $a=0$, a contradiction.
\end{proof}
By Theorem \ref{thm44} $(\Delta,T)$ does have the norm preserving extension
property. One can also write down a simple formula to see this.
If $f$ is given by \eqref{eq41} and  has norm one and is not constant, let $a = f(0)$.
Then the function
\[
F(\l_1, \l_2) \ =\ m_a \left( m_a(f(\l_1,0) + m_a(f(0,\l_2) \right) 
\]
is a norm-preserving extension of $f$.

The singularity in $T$ does not prevent the existence of np pairs $(G,T)$ with the  linear
 norm preserving extension property.
\begin{thm}
There is a domain of holomorphy $G$ containing $T$ so that $(G,T)$ has the linear
 norm preserving extension property.
\end{thm}
\begin{proof}
Define a linear extension operator $E$ by
\[
E(\Psi_{f_1, f_2}) (\l) \ = \ f_1(\l_1) + f_2(\l_2) - f_1(0) .
\]
Define $G$ by
\begin{equation}
\label{eq54}
\notag
G \ = \ \left[ \bigcap_{f \in \hinf(T), \ f \ {\rm non-constant} } \{ \l \in \c^2 | \ | E \circ f (\l) | < \| f \|_T \} 
\right]^\circ .
\end{equation}
By \cite[Cor. II.3.19]{ran}, $G$ is a domain of holomorphy, provided it is non-empty. So we need to prove that
$T \subset G$.
Define $H \subset \d^2$ by
\[
H \ = \ \{
\frac{|\l_2|}{1-|\l_2|} < \frac{1}{2}\frac{1 - |\l_1|}{1+|\l_1|} \}
\cup \{ \frac{|\l_1|}{1-|\l_1|} < \frac{1}{2}\frac{1 - |\l_2|}{1+|\l_2|} \} .
\]
We will show that $H \subseteq G$, and as $T \subset H$, we will be done.

Choose $\l \in H$, and assume without loss of generality that 
 it satisfies the first inequality. Let $f = \Psi_{f_1,f_2}$ be a non-constant function
in $\hol(T)$, and assume $\| f \|_T = 1$. Let $c = |f(0)|$. Then
\begin{eqnarray*}
|E \circ f (\l) | & \ \leq \ & | f_1(\l_1) | + |f_2(\l_2) - f_2(0) |\\
&\leq &
 \frac{c + |\l_1|}{1 + |\l_1| c}
+
\frac{|\l_2|}{1-|\l_2|}(1 - c^2) \\
&\leq &
\frac{c + |\l_1|}{1 + |\l_1| c}
+
\frac{1}{2}\frac{1 - |\l_1|}{1+|\l_1|}
(1 - c^2) \\
& < &
\frac{c + |\l_1|}{1 + |\l_1| c}
+
\frac{1 - |\l_1|}{1+|\l_1|}
(1 - c) \\
& < &
\frac{c + |\l_1|}{1 + |\l_1| c}
+
\frac{1 - |\l_1|}{1+|\l_1|c}
(1 - c) \\
&=& 1.
\end{eqnarray*}
In the second line we used Lemma \ref{lem55}, and in the third the defining inequality for $H$.
We have shown that every point $\l$ in $H$ satisfies $|E \circ f (\l) | < \| f \|_T$ for every non-constant
function in $\hinf(T)$, and since $H$ is open we conclude that $H \subseteq G$, as desired.
\end{proof}

The following lemma is a straightforward consequence of the Schwarz-Pick lemma.

\begin{lem} \label{lem55}
Suppose $g: \d \to \d$ is holomorphic. Then for all $z \in \d$
\begin{eqnarray}
\label{eq56}
\notag
 | g(z) | & \ \leq \ & \frac{|g(0)| + |z|}{1 + |z| |g(0)|} \\
|g(z) - g(0)| & \leq & \frac{|z|}{1-|z|}(1 - |g(0)|^2) .
\label{eq57}
\notag
\end{eqnarray}
\end{lem}

\begin{ques}
Is there a Theorem \ref{exist.thm.20} for the linear 
 norm preserving extension property?
\end{ques}

\section{The Representation of Even Schur Functions on the Bidisc}
\label{sec6}

For use in Section \ref{sec7}, we derive a model and a representation 
for even Schur functions on the bidisc.

A {\em decomposed Hilbert space}
 is a Hilbert space $\m$ together with two orthogonal subspaces
$\m_1$ and $\m_2$ satisfying  $\m=\m_1 \oplus \m_2$.
By \cite{ag90}, a function $\phi: \d^2 \to \c$  is in $\mathscr{S}(\d^2)$ if 
 and only if there exist a separable decomposed Hilbert space $\m=\m_1 \oplus \m_2$,
and a 
 function $v:\d^2 \to \m$ such that
\begin{equation}
\label{eq61}
1-\overline{\phi(\mu)}\phi(\lambda) = \ip{(1- \mu^* \lambda) v_\lambda}{v_\mu}.
\end{equation}
In \eqref{eq61} we write $v_\lambda$ for $v(\l)$, and the operator $\l : \m \to \m$ 
is defined by 
\be
\label{eq611}
\l \ = \ \l_1 I_{\m_1} \oplus \l_2 I_{\m_2} .
\ee
We call \eqref{eq61} a {\em model} for $\phi$.

We say $\phi \in \hol(\d^2)$ is \emph{even} if
\[
\forall_{\lambda \in \d^2}\ \ \phi(-\lambda)=\phi(\lambda).
\]
We let $\hol_{\, \rm even}\ (\d^2)$ denote the even holomorphic functions on $\d^2$,
and  $\mathscr{S}_e(\d^2)$ the ones in Schur class.
\begin{prop}
\label{prop62}
Let $\phi:\d^2 \to \c$. Then $\phi$ is  in $\mathscr{S}_e(\d^2)$ if and only if there exist a separable decomposed Hilbert space $\m=\m_1 \oplus \m_2$, a unitary operator $U$ acting on $\m$, and 
 an even
 function $u:\d^2 \to \m, \l \mapsto u_\l,$ such that
\be
\label{eq63}
1-\overline{\phi(\mu)}\phi(\lambda) = \ip{(1-(\mu U \mu)^* (\lambda U\lambda))u_\lambda}{u_\mu}.
\ee
\end{prop}
\begin{proof}
 Note that
 %\green
 if  \eqref{eq63}  holds, then it 
 defines $\phi$ 
(up to a unimodular constant), and it is clear that $\phi$ is even and bounded by $1$ in modulus.
So to see it is in $\mathscr{S}(\d^2)$, we must show it is holomorphic. 
This follows from the argument in the proof of \cite[Proposition 4.26]{amy20}, which shows
that $\lambda \mapsto u_\lambda$ is holomorphic.
%\black

Conversely, suppose that $\phi$ is an even function in $\mathscr{S}(\d^2)$. Then there is some
model so that \eqref{eq61} holds. Define
\begin{eqnarray*}
v_e(\l) & \ = \ & \frac{1}{2} ( v(\l) + v(-\l) ) \\
v_o(\l) &=&  \frac{1}{2} ( v(\l) - v(-\l) )
\end{eqnarray*}
A lurking isometry argument (see \cite[Sec. 2.4]{amy20}) shows that there is a partial isometry  $W : \c \oplus \m \oplus \m \to \c \oplus \m \oplus \m$ 
satisfying
\be
\label{eq64}
W :
\begin{pmatrix}
1 \\
\l v_o(\l) \\
\l v_e(\l)
\end{pmatrix} \
\mapsto \
\begin{pmatrix}
\phi(\l) \\
 v_e(\l) \\
 v_o(\l)
\end{pmatrix}
\ee
Expanding both sides of \eqref{eq64} in powers of $\l$, we see that $W$ has to map the odd part to the odd part, so 
\be
\label{eq65}
W \begin{pmatrix}0 \\ 0 \\ \l v_e(\l)) \end{pmatrix} =\begin{pmatrix} 0 \\ 0 \\ v_o(\l)  \end{pmatrix}. 
\ee 
Thus we see that $U(\l v_e(\l)) = v_o (\l), $ for some partial isometry $U:\m\to \m$.
By adding a separable infinite dimensional Hilbert space to $\m$ if necessary, one can additionally assume that the
partial isometry $U$ is a unitary.

Define $u_\l = v_e (\l)$. Substituting this into \eqref{eq64} and using \eqref{eq65}, we get
\be
\label{eq66}
W :
\begin{pmatrix}
1 \\
\l U \l u_\l \\
\l u_\l
\end{pmatrix} \
\mapsto \
\begin{pmatrix}
\phi(\l) \\
u_\l \\
 U \l u_\l
\end{pmatrix}
\ee
Again we expand \eqref{eq66} into odd and even parts, and get
\be
\label{eq67}
W :
\begin{pmatrix}
1 \\
\l U \l u_\l \\0
\end{pmatrix} \
\mapsto \
\begin{pmatrix}
\phi(\l) \\
u_\l \\ 0
\end{pmatrix}
\ee
Since $W$ is  a partial isometry, we get that for any $\l,\mu$ in $\d^2$ we have
\be
\notag
\Big\langle
\begin{pmatrix}
1 \\
\l U \l u_\l \\ 0
\end{pmatrix} ,
\begin{pmatrix}
1 \\
\mu U \mu u_\mu \\0 
\end{pmatrix} \Big\rangle_{\c\oplus \m \oplus \m}
\ = \ \Big\langle
\begin{pmatrix}
\phi(\l) \\
u_\l \\0
\end{pmatrix} ,
\begin{pmatrix}
\phi(\mu) \\
u_\mu \\0
\end{pmatrix} \Big\rangle_{\c\oplus \m \oplus \m},
\ee 
which rearranges into \eqref{eq63}.
\end{proof}

Let $\m$ be a Hilbert space. We define $\ball[ \b (\m)]$ by
\[
\ball[ \b(\m)] = \set{X\in \b(\m)}{\norm{X}<1}.
\]
We let $\mathscr{R}(\m)$ denote the set of 4-tuples  $\xi=(a,\beta, \gamma,D)$ such that $a\in \c$, $\beta \in \m$, $\gamma \in \m$,  $D\in \b(\m)$, and the block operator $L_\xi \in \b(\c \oplus \m)$ defined by
\be
\label{eq68}
L_\xi = \begin{bmatrix}
          a & 1\otimes \beta\\
          \gamma \otimes 1 & D
        \end{bmatrix}
\ee
is unitary.
If  $\xi\in \mathscr{R}(\m)$, then we may define a function $F_\xi$ on $\ball[\b(\m)]$ by the formula
\be
\label{eq681}
F_\xi(X) = a+\ip{X(1-DX)^{-1}\gamma}{\beta}_\m,\qquad X\in \ball[\b(\m)]
\ee
\begin{prop}
\label{prop69}
$\phi \in \ess_e(\d^2)$ if and only if there exists a separable decomposed Hilbert space  
 $\m=\m_1 \oplus \m_2$, a unitary $U\in \b(\m)$, and $\xi \in \R(\m)$ such that
\[
\forall_{\lambda \in \d^2}\ \ \phi(\lambda)=F_\xi(\lambda U \lambda).
\]
\end{prop}
\begin{proof}
This follows from Proposition \ref{prop62},
where we write $U$ as $L_\xi$  in \eqref{eq68}, and then use \eqref{eq67} to solve for
$\phi$.
\end{proof}
Propositions~\ref{prop62} and \ref{prop69} easily generalize to Schur functions $\phi$ that
satisfy $\phi(\omega \l) = \phi(\l)$, where $\omega$ is a primitive $n^{\rm th}$ root of unity.
One just replaces $\l U \l$ by $\l (U \l)^{n-1}$.

%\green
%More interesting is if $\phi$ is invariant under some group of automorphisms that is more complicated. Perhaps work this out later?
%\black

\section{The Representation of Schur Functions on $\calv$}
\label{sec7}

We define a variety in $\d^3$ by
\be
\label{eq71}
\calv=\set{z \in \d^3}{z_3^2 =z_1z_2}
\ee
and let $\pi:\d^2 \to \calv$ denote the surjective map defined by
\[
\pi(\lambda)=(\lambda_1^2,\lambda_2^2,\lambda_1\lambda_2).
\] 
The map $\pi$ is a two-to-one branched cover of $\d^2$ over $\calv$. 
The key observation is that 
\be\label{eq72}
\Phi\in \ess(\calv)\iff \Phi \circ \pi\in \ess_e(\d^2).
\ee

For a separable decomposed Hilbert space $\m=\m_1\oplus \m_2$ and a unitary operator $U\in\b(\m)$ we may decompose $U$,
\[
U=\begin{bmatrix}
    A & B \\
    C & D
  \end{bmatrix},
\]
where $A:\m_1 \to \m_1$, $B:\m_2 \to \m_1$, $C:\m_1 \to \m_2$, and $D:\m_2 \to \m_2$. Furthermore, for each $z\in \c^3$, we may define an operator $z_U$ by the formula
\[
z_U =\begin{bmatrix}
    Az_1 & Bz_3 \\
    Cz_3 & Dz_2
  \end{bmatrix},\qquad z\in \c^3.
\]
We then have that if we view $\lambda$ as an operator on $\m$ (as in \eqref{eq611}) and $z=\pi(\lambda)$, then
\[
\lambda U \lambda = \begin{bmatrix}
    A\lambda_1^2 & B\lambda_1\lambda_2 \\
    C\lambda_1\lambda_2 & D\lambda_2^2
    \end{bmatrix}
    =\begin{bmatrix}
    Az_1 & Bz_3 \\
    Cz_3 & Dz_2
  \end{bmatrix}
  =z_U
\]
Consequently, using Propositions \ref{prop62} and \ref{prop69} we obtain the following description of the Schur class of $\calv$.
\begin{prop}
\label{prop72}
$\Phi \in \ess(\calv)$ if and only if there exists a decomposed Hilbert space  $\m$, a unitary $U\in \b(\m)$, and $\xi \in \R(\m)$ such that
\[
\forall_{\lambda \in \d^2}\ \ \Phi(\lambda)=F_\xi(z_U).
\]
\end{prop}
Let $\mathscr{U}$ denote the collection of ordered pairs $(\m,U)$ such that $\m$ is a separable decomposed Hilbert space and $U$ is a unitary operator acting on $\m$ and define $\calg \subseteq \c^3$ by ordaining that $z \in \calg$ precisely when
\[
\norm{z_U}<1
\]
for all separable decomposed Hilbert spaces $\m$ and all unitary $U$ acting on $\m$. Thus,
\be\label{rep.10}
z \in \calg \iff \forall_{(\m,U) \in \u}\ \ \norm{z_U} <1.
\ee

It is clear from the definition that $\calg$ is convex, and it
 follows from 
 %Lemma~\ref{env.lem.10}
 %\green
 Lemma~\ref{lem40}
 %\black
  that $\calg$ is open.
\begin{prop}
\label{prop72}
If $\Phi \in \ess(\calv)$, then there exists $\Phi^\sim \in \ess(\calg)$ such that $\Phi=\Phi^\sim|\calv$.\emph{}
\end{prop}
\begin{proof}
Since $\Phi \in \ess(\calv)$ it has a realization as in Proposition \ref{prop72}.
Define $\Phi^\sim$ by $\Phi^\sim (z) = F_\xi(z_U).$
It is immediate from \eqref{eq681} that this defines an analytic function;
to show it is in the Schur class requires a calculation. See eg \cite[Sec. 3.9]{amy20}.

\end{proof}

\section{The Envelope}
\label{sec8}

In this section we compute $\calg$ defined by \eqref{rep.10}. For $z\in \c^3$, define $\norm{z}_\u$ by
\[
\norm{z}_\u = \sup_{(\m,U) \in \u}\norm{z_U}.
\]
\begin{lem}\label{env.lem.10}
If $z\in \c^3$, then
\[
z \in \calg \iff \norm{z}_\u <1.
\]
\end{lem}
\begin{proof}
The lemma will follow if it can be shown that when $z\in \c^3$, there exists $(\m^\sim,U^\sim)\in \u$ such that
\be\label{env.10}
\norm{z_{U^\sim}}=\sup_{(\m,U) \in \u}\norm{z_U}.
\ee
Choose a sequence $\{(\m^{(n)},U^{(n)})\}$ in $\u$ such that
\[
\norm{z_{U^{(n)}}} \to \sup_{(\m,U) \in \u}\norm{z_U}\ \  \text{ as }\ \  n\to \infty.
\]
Let
\[
\m^\sim=\bigoplus_n \m^{(n)}\qquad \text{ and }\qquad U^\sim =
 \bigoplus_n U^{(n)}.
\]
if we decompose $\m^\sim$,
\[
\m^\sim = \Big(\bigoplus_n (\m_1^{(n)} \oplus 0)\Big)\  \oplus\
\Big(\bigoplus_n (0 \oplus \m_2^{(n)})\Big),
\]
then $(\m^\sim,U^\sim) \in \u$, and as
\[
z_{U^\sim}\  \cong\ \  \bigoplus_n z_{U^{(n)}},
\]
\eqref{env.10} holds.
\end{proof}

We let $\mathscr{C}$ denote the collection of ordered pairs $(\m,T)$ where $\m=\m_1\oplus \m_2$ is a decomposed Hilbert space and $T$ is a contraction acting on $\m$. If $(\m,T) \in \mathscr{C}$, we may decompose $T$,
\be\label{env.20}
T=\begin{bmatrix}
    A & B \\
    C & D
  \end{bmatrix}
\ee
and then for $z\in \c^3$ define $z_T \in \b(\m)$ by
\[
z_T=\begin{bmatrix}
    Az_1 & Bz_3 \\
    Cz_3 & Dz_2
  \end{bmatrix}
\]
We let $\C_2$ denote the collection of ordered pairs $(\m,T) \in \C$ satisfying $\dim \m_1 =\dim \m_2 = 2$.
\begin{lem}\label{env.lem.20}
If $z \in \c^3$, then
\[
\norm{z}_\u \ge \sup_{(\m,T) \in \C_2}\norm{z_T}.
\]
\end{lem}
\begin{proof}
Fix $(\m,T) \in \C_2$ and assume that $T$ is represented as in \eqref{env.20} with respect to the decomposition $\m=\m_1\oplus \m_2$. By the Sz.-Nagy Dilation Theorem \cite{szn53} there exists a decomposed Hilbert space
\[
\h =\h_0 \oplus \m_1 \oplus \m_2 \oplus \h_1
\]
on which the block operator $U$ defined on $\h$ by
\[
U=\begin{bmatrix}
    X_{11} & X_{12} & X_{13} & X_{14} \\
    0 & A & B & X_{24} \\
    0 & C & D & X_{34} \\
    0 & 0 & 0 & X_{44}
  \end{bmatrix}
\]
is unitary.
If we decompose $\h$ as
\[
\h =(\h_0 \oplus \m_1) \oplus (\m_2 \oplus \h_1),
\]
then
\[
z_U=\begin{bmatrix}
    X_{11}z_1 & X_{12}z_1 & X_{13}z_3 & X_{14}z_3 \\
    0 & Az_1 & Bz_3 & X_{24}z_3 \\
    0 & Cz_3 & Dz_2 & X_{34}z_2 \\
    0 & 0 & 0 & X_{44}z_3
  \end{bmatrix}
\]
and we see that
\[
\norm{z_T}=\norm{\begin{bmatrix}
    Az_1 & Bz_3 \\
    Cz_3 & Dz_2
  \end{bmatrix}}
  \le \norm{z_U} \le \norm{z}_\u.
\]
As $(\m,T)$ is an arbitrary element of $\C_2$, this proves the lemma.
\end{proof}
\begin{lem}\label{env.lem.30}
If $z \in \c^3$, then
\[
\norm{z}_\u \le \sup_{(\m,T) \in \C_2}\norm{z_T}.
\]
\end{lem}
\begin{proof}
Fix $z \in \c^3$ and $\epsilon >0$. Choose $(\m,U) \in \u$ such that
\be\label{env.30}
\norm{z}_\u -\epsilon/2 < \norm{z_U}
\ee
and choose $\gamma \in \m$ such that $\norm{\gamma}=1$ and
\be\label{env.40}
\norm{z_U} -\epsilon/2 <\norm{z_U \gamma}.
\ee
If we let $P_1$ and $P_2$ denote the orthogonal projections of $\m$ onto $\m_1$ and $\m_2$, respectively, and define
\[
\n_1 =\spn \{P_1\gamma, P_1 z_U\gamma\}\ \  \text{ and }\ \  \n_2 =\spn \{P_2\gamma, P_2 z_U\gamma\},
\]
 then $\n_1 \perp \n_2$, i.e., $\n=\n_1 +\n_2$ is a decomposed Hilbert space. Furthermore, if we define $T\in \b(\n)$ by letting $T=P_\n U|\n$, then $(\n,T) \in \C$, $\gamma \in \n$, and
 \[
z_T \gamma =z_U \gamma.
 \]
 Hence, using \eqref{env.30} and \eqref{env.40} it follows that
 \be\label{env.50}
 \norm{z}_\u -\epsilon < \norm{z_T \gamma} \le \norm{z_T}.
 \ee

Now, it might be the case that either $\dim \n_1$ or $\dim \n_2$ is strictly less than 2, i.e., $(\n,T)\not\in \C_2$. However, we may choose Hilbert spaces $\l_1$ and $\l_2$ such that $\dim (\l_1 \oplus \n_1) = 2$ and $\dim (\n_2 \oplus \l_2) = 2$ and then define $(\n^\sim,T^\sim) \in \C_2$ in the following way. Let
\[
\n^\sim = (\l_1 \oplus \n_1) \oplus (\n_2 \oplus \l_2),
\]
and define $T^\sim$ by choosing contractions $L_1\in \b(\l_1)$ and $L_2 \in \b(\l_2)$ and letting
\[
T^\sim \big((x \oplus u) \oplus (v\oplus y)\big)= L_1 x \oplus (T (u \oplus v)) \oplus L_2 y,
\qquad x \in \l_1, y\in \l_2,  u \in \n_1,v \in \n_2.
\]
Here, we have identified the two spaces
\[
(\l_1 \oplus \n_1) \oplus (\n_2 \oplus \l_2)\ \ \text{ and }\ \
\l_1 \oplus (\n_1 \oplus \n_2) \oplus \l_2.
\]
With this definition it follows that
\[
z_{T^\sim} \cong (z_1L_1) \oplus z_T \oplus z_2 L_2.
\]
Therefore,
\[
\norm{z_T} \le \norm{z_{T^\sim}},
\]
which implies via \eqref{env.50} that
\be\label{env.60}
 \norm{z}_\u -\epsilon < \norm{z_{T^\sim}}.
\ee

Summarizing, we have shown that if $\epsilon>0$, then there exists $(\n^\sim,T^\sim) \in \C_2$ such that \eqref{env.60} holds. This proves the lemma.
\end{proof}
Finally we show that to check if $z \in \calg$, you just need to check that $\| z_U \| < 1$ for every unitary on 
a decomposed Hilbert space of dimension $2+2$.
We  let $\u_2$ denote the collection of ordered pairs $(\m,U) \in \u$ satisfying $\dim \m_1 =\dim \m_2 = 2$. 
\begin{prop}
\label{prop810}
If $z \in \c^3$, then
\[
\norm{z}_\u = \sup_{(\m,U) \in \u_2}\norm{z_U}.
\]
\end{prop}
\begin{proof}
Fix a 2 dimensional Hilbert space $\n$, let $\m=\n\oplus \n$, and let $\ball \b(\m)$ denote the closed unit ball of $\b(\m)$. By Lemmas \ref{env.lem.20} and \ref{env.lem.30}
\[
\norm{z}_\u = \sup_{T\in \ball \b(\m)}\norm{z_T}.
\]
But the map $T \mapsto \norm{z_T}$ is convex and therefore attains its maximum at an extreme point of $\ball \b(\m)$. As the extreme points of $\ball \b(\m)$ are the unitaries, it follows that there exists $(\m,U) \in \u_2$ such that $\norm{z_U} = \norm{z}_\u$.
\end{proof}
Now we shall derive an inequality that defines $\calg$.

\begin{lem}
\label{lem.10}
$
\begin{bmatrix}
A&B\\C&D
\end{bmatrix}$
 is a block unitary acting on $ \c^2 \oplus \c^2$ if and only if
 there exist 3 unitary $2\times 2$ matrices $u,v,w$ such that
\be\label{35}
\begin{bmatrix}
A&B\\C&D
\end{bmatrix}=
\begin{bmatrix}
u&0\\0&v
\end{bmatrix}
\begin{bmatrix}
a&b\\c&d
\end{bmatrix}
\begin{bmatrix}
1&0\\0&w
\end{bmatrix}
\ee
where
\be\label{40}
\begin{bmatrix}
a&b\\c&d
\end{bmatrix}
\text{ is a block unitary acting on } \c^2 \oplus \c^2
\ee
and
\be\label{50}
a,b,c \ge 0.
\ee
\end{lem}
\begin{proof}
 Using polar decomposition there exist unitary $u,v$ and $a,c\ge 0$ such that
\[
A=ua\ \text { and }\  C=vc.
\]
Again using polar decomposition, there exists a unitary $w$ and $b\ge 0$ such that
\[
B=ubw.
\]
If we set $d=v^*Dw^*$, then \eqref{35} holds by direct  computation, and, as the product of unitaries is unitary, \eqref{40} holds as well.
The converse is immediate.
\end{proof}
\begin{lem}\label{lem.20}
\eqref{40} and \eqref{50} hold if and only if there exists a unitary $2\times 2$ matrix $u$ and scalars $r\le s$ in $[0,1]$ such that
\be\label{60}
\begin{bmatrix}
u^*&0\\0&u^*
\end{bmatrix}
\begin{bmatrix}
a&b\\c&d
\end{bmatrix}
\begin{bmatrix}
u&0\\0&u
\end{bmatrix}
=
\begin{bmatrix}
r&0&\sqrt{1-r^2}&0\\
0&s&0&\sqrt{1-s^2}\\
\sqrt{1-r^2}&0&x_{11}&x_{12}\\
0&\sqrt{1-s^2}&x_{21}&x_{22}
\end{bmatrix}
\ee
where
\be\label{70}
\begin{bmatrix}
x_{11}&x_{12}\\x_{21}&x_{22}
\end{bmatrix}
=
\threepartdef{
\begin{bmatrix}
-r&0\\0&-s
\end{bmatrix}}{s<1,}
{\begin{bmatrix}
-r&0\\0&\tau
\end{bmatrix},\text{ where } |\tau|=1,} {r<s=1,}
{V,\text{ where $V$ is unitary, }}{r=s=1.}
\ee
\end{lem}
\begin{proof}
By the spectral theorem there exists a unitary $2\times 2$ matrix $u$ and scalars $r\le s$ in $[0,1]$ such that
\[
u^*au =
\begin{bmatrix}
r&0\\0&s
\end{bmatrix}
\]
As \eqref{40} holds we have that
\be\label{75}
\begin{bmatrix}
a&c\\b&d^*
\end{bmatrix}\begin{bmatrix}
a&b\\c&d
\end{bmatrix}=
\begin{bmatrix}
1&0\\0&1
\end{bmatrix}=
\begin{bmatrix}
a&b\\c&d
\end{bmatrix}
\begin{bmatrix}
a&c\\b&d^*
\end{bmatrix}
\ee
In particular, $a^2+c^2=1$. Hence, as $c\ge 0$, $c=\sqrt{1-a^2}$. But then
\[
u^*cu=u^*(\sqrt{1-a^2})u=\sqrt{1-(u^*au)^2}=\begin{bmatrix}
\sqrt{1-r^2}&0\\0&\sqrt{1-s^2}.
\end{bmatrix}
\]
Likewise, as $a^2 +b^2 =1$ and $b\ge0$,
\[
u^*bu=\begin{bmatrix}
\sqrt{1-r^2}&0\\0&\sqrt{1-s^2}.
\end{bmatrix}
\]
There remains to show that $d$ has the form described in the statement of the lemma. This follows by simple calculations using the relations
\[
b^2+d^*d= 1= c^2 +dd^*,\ \ ab+cd =0,\ \ \text{ and }\ \  ca +db=0.
\]
\end{proof}
 Combining Proposition \ref{prop810} with Lemmas \ref{lem.10} and \ref{lem.20} yields the following lemma. 
\begin{lem}\label{lem.30}
The point $z \in \d^3$ is in $\calg$
if and only if
\be\label{80}
\sup_{r\in[0,1]}\big\| z_r\big\| < 1
\ee
where
\[
z_r=\begin{bmatrix}
rz_1&\sqrt{1-r^2} z_3\\ \sqrt{1-r^2} z_3&-rz_2
\end{bmatrix}.
\]
\end{lem}
\begin{proof}
For $U = 
\begin{bmatrix}
A&B\\C&D
\end{bmatrix}$, we have by Lemma \ref{lem.20} that
$z_U$ is unitarily equivalent to $z_r \oplus z_s$ when $ s < 1$,
so \eqref{80} is necessary. It is easy to check the cases when $s=1$ to see that
the condition is also sufficient.
\end{proof}

\begin{lem}
\label{lem40} The point
$z \in \d^3$ is in $\calg$ if and only if 
\be
\label{eq:38}
|z_1 z_2 - z_3^2| < (1-|z_3|^2) + \sqrt{1-|z_1|^2}\sqrt{1-|z_2|^2}.
\ee
\end{lem}
\begin{proof}
We need to find those $z$ such that spectral radius of $z_r z_r^*$ is less then one for every $r\in [0,1].$
Letting $t = r^2$, we get that the trace $\tau$ of  the matrix $z_r z_r^*$ is
 $$\tau = t |z_1|^2 + t |z_2|^2 + 2 (1-t) |z_3|^2$$ 
 (this must be in the interval $(0,2)$)
 and the determinant $d$ is $$d= |t z_1 z_2 + (1-t) z_3^2|^2.$$ 
For the spectral radius to be less than one, we need to have $\tau + \sqrt{\tau^2 -4 d } < 2$. Since $\tau < 2$, this
  is equivalent to $\tau < 1+ d$. So we want
$$t |z_1|^2 + t |z_2|^2 + 2(1-t)|z_3|^2 < 1 + t^2 |z_1 z_2|^2 + (1-t)^2 |z_3|^4 + 2 t (1-t) \Re z_1 z_2 \bar z_3^2.
$$
Simple calculations show that this inequality is equivalent to
\begin{equation}\label{eq:1}
t^2|z_1 z_2 - z_3^2|^2 - t \left(|z_1 z_2 - z_3^2|^2 + (1-|z_3|^2)^2 - (1-|z_1|^2) (1- |z_2|)^2\right) + (1- |z_3|^2)^2>0.
\end{equation}
We want \eqref{eq:1} to be satisfied for every $t\in (0,1)$. Let
$A=|z_1 z_2 - z_3^2|$, $B= (1-|z_3|^2)$ and $C=\sqrt{1-|z_1|^2}\sqrt{1-|z_2|^2}$,
 and $a:=A^2$, $b:= B^2$, $c:=C^2$, $a\geq 0$, $b,c>0$.
Note that 
\be
\label{eq:39}
a t^2 - (a+b-c) t + b >0\ \forall \  t\in [0,1]
\ee
 if either $a+b -c \leq 0$ or $(a+b-c)/2a \geq 1$ or $(a+b-c)^2 < 4 a b$.
 (The inequality is automatically satisfied when $t=0$ or $1$.) 

The first equality means that $a\leq  c-b$. The second one that $a\leq b-c$ (so both together mean that $a\leq |b-c|$). The third one can be transformed to $$(a-b-c)^2 < 4bc,$$ that is $|A^2 - B^2 - C^2|< 2 BC$, which means that $|B-C|<A< B+C$.

Summing up, either $A^2\leq |B^2 - C^2|$ or $|B-C|<A<B+C$. Since $|B-C|^2\leq |B^2 - C^2|$ we get that 
\eqref{eq:39} holds if and only if $A<B+C$, that is
\eqref{eq:38}.
\end{proof}
Putting all these results together, we get the following theorem.
\begin{thm}
\label{thm814}
Let $\calv$ and $\calg$ be defined by 
\begin{eqnarray*}
\calv &\ =\ & \set{z \in \d^3}{z_3^2 =z_1z_2}\\
\calg  &\ =\ & \set{z \in \d^3}{|z_1 z_2 - z_3^2| < (1-|z_3|^2) + \sqrt{1-|z_1|^2}\sqrt{1-|z_2|^2}} .
\end{eqnarray*}
Then $\calg$ is convex, and 
 $(\calg,\calv)$ is an np pair. 
\end{thm}

Now we shall show that $\calg$ is the largest balanced set for which $(G,\calv)$ is np.
We exploit the fact that even though $\calv$ is two-dimensional, it effectively has a 3-dimensional
tangent space at $0$.

\begin{lem}
\label{lem826}
Assume that $(G, \calv)$ is a Cartan pair.
Let $\psi \in \hol(G)$, let $f = \psi |_G$, and let 
 $F \in \hol(G)$ be any holomorphic extension of $f$ to $G$.
Then $D F (0) = D\psi(0)$.
\end{lem}
\begin{proof}
As $F - \psi$ vanishes on $\calv$, we have 
\be
\label{eq827}
F(z) - \psi(z) \ = \
(z_3^2 - z_1z_2) h(z)
\ee
for some $h \in \hol(G)$ \cite[p. 27]{chi90}.
As the derivative of the right-hand side of \eqref{eq827} 
vanishes at $0$, we get the result.
\end{proof}

\begin{thm}
\label{thm828}
Let $G$ be a balanced domain in $\c^3$, and assume that $(G, \calv)$ is an np pair.
Then $G \subseteq \calg$.
\end{thm}
\begin{proof}
Suppose that $G$ is not a subset of $\calg$. Since $G$ is open, this means
we can find a point $\l \in G$ that is not in $\overline{\calg}$.
By Lemma \ref{env.lem.10} and Proposition \ref{prop810}, this means that there is a pair
$(M,U) \in \u_2$ so that $\| \l_U \| > 1 $.
This means there are unit vectors $\xi, \eta \in \c^4$ so that the linear function
\[
f(z) \ = \ \langle z_U \xi, \eta \rangle
\]
is larger than $1$ when $z = \l$.
But by Lemma \ref{env.lem.10}, we have $|f|$ is less than $1$ on $\calg$, and hence on $\calv$.

Let $F$ be a norm-preserving extension of $f$ to $\ess(G)$.
Define a function of one variable $g(\zeta) = F(\zeta \l)$.
Since $G$ is balanced, $g$ is defined on $\d$,  and it is in the Schur class since $F$ is.
Moreover $g(0) = 0$ and by Lemma \ref{lem826}
\[
g'(0) = D F (0) \begin{bmatrix} \l_1\\ \l_2 \\ \l_3 \end{bmatrix}
=  D f (0) \begin{bmatrix} \l_1\\ \l_2 \\ \l_3 \end{bmatrix} 
= f(\l) .
\]
This contradicts Schwarz's lemma.
\end{proof}

\section{The Noncommutative Analysis Setting}
\label{sec9}

\subsection{Overview}
Joe Taylor introduced noncommutative analysis in his seminal work \cite{tay73} on functional calculus for noncommuting elements of a Banach algebra. Landmarks in the development of the theory are Dan Voiculescu's works \cite{voi04,voi10} in the context of developing the theory of free probability, Bill Helton's result \cite{helt02} proving that positive free polynomials are sums of squares, the advance of Helton and McCullough \cite{hm12} as a step in  Helton's fruitful program to develop a descriptive theory of the domains on which LMI and semi-definite programming apply, and the recent monograph by Kaliuzhnyi-Verbovetzkii and Vinnikov
\cite{kv14} that gives a panoramic view of the field to date, and establishes
a beautiful ``Taylor-Taylor'' formula for nc-functions.

There is much other notable work as well. As a sampling, there are articles of Popescu \cite{po06, po08, po10, po11} which extend various results from classical function theory to functions of d-tuples of bounded operators; the magnum opus \cite{bgm} of Ball,Groenewald and Malakorn, which extends realization formulas for functions of commuting operators to functions of non-commuting operators; Alpay and Kalyuzhnyi-Verbovetzkii \cite{alpkal} which studies realization formulas for noncommutative rational functions that are J-inner; and \cite{hkm11a,hkm11b} where  Helton, Klep and
McCullough study mappings of noncommutative domains.
 \\ \\

In this section, which is largely expository in nature, we shall present some of the techniques developed 
in the papers \cite{agmc15a}, \cite{agmc15b}, and \cite{agmc13b} to prove the existence of free holomorphic extensions of both holomorphic functions defined on varieties in $\c^d$ and free holomorphic functions defined on free varieties in $\m^d$, the ``$d$ dimensional noncommutative universe''. As an application one obtains a classical Cartan theorem with sharp bounds,
 with norms that are defined using matrices.

\subsection{Free Holomorphic Functions}
There are many types of analysis that can be carried out on functions in noncommuting variables. Here, we shall focus on \emph{free analysis}, by which we mean the study of \emph{free holomorphic functions}.
Let $\pd$ denote the free algebra on $d$ generators and let $\mn$ denote the space of $n\times n$ matrices with complex entries. Let $\mn^d$ denote the space of $d$-tuples of $n \times n$ matrices and define $\md$, the \emph{$d$-dimensional nc universe}, by
\[
\md = \bigcup_{n=1}^\infty \mn^d.
\]
We may equip $\md$ with the \emph{coproduct topology} wherein one ordains that a set $D$ in $\md$ is open if and only if $D \cap \mn^d$ is open in $\mn^d$ for each integer $n\ge 1$.

If $\delta \in \pd$ and $x\in \md$ then we may form $\delta(x)$ in the natural way. More generally, if $\delta=[\delta_{ij}]\in \pdij$, the collection of $I\times J$ matrices with entries in $\pd$ and $x \in \md$, we may form
\[
\delta(x) = [\delta_{ij}(x)].
\]

We say that a subset of $\md$ is \emph{basic} if it has the form
\[
B_\delta=\set{x\in \md}{\norm{\delta(x)}<1}
\]
for some matrix $\delta$ with entries in $\pd$. Noting that
\[
B_{\delta_1} \cap B_{\delta_2} = B_{\delta_1 \oplus \delta_2},
\]
it follows that the collection of basic sets forms a basis for a topology on $\md$, which we refer to as the \emph{free topology}. We say that a subset of $\md$ is a \emph{domain} if it is open in the free topology.

If $D \subseteq \md$ is a free domain, and $f:D \to \mone$ is a function, we say that $f$ is a \emph{free holomorphic function} if $f$ can be locally uniformly approximated by free polynomials, i.e., for each $\lambda \in D$ there exists a basic set $B_\delta$ such that
\[
\lambda \in B_\delta \subseteq D
\]
and
\[
\forall_{\epsilon>0}\  \exists_{\pi \in \pd}\  \forall_{x\in B_\delta}\ \  \norm{f(x)-\pi(x)} < \epsilon.
\]

If $D$ is a free domain and $E\subseteq D$, we shall refer to $(D,E)$ as a \emph{free pair}. If $(D,E)$ is a free pair,  then a function $f:E \to \mone$ is said to be \emph{free holomorphic on $E$} if for each point $\lambda\in E$ there exist a basic set $B_\delta$ and a free holomorphic function $F$ on $B_\delta$ such that
\[
\lambda \in B_\delta \subseteq D\qquad \text{ and }\qquad
F|E =f.
\]
Finally, we say that a free pair $(D,E)$ is \emph{norm preserving} if for each bounded free holomorphic function $f$ on $E$,  there exists a free holomorphic function $F$ on $D$ that extends $f$ and such that
\[
\sup_{x \in D} \norm{F(x)} = \sup_{x \in E} \norm{f(x)}.
\]
\subsection{Pick Pairs}
If $S\subseteq \pd$, we define $V(S)$ in $\md$ by
\[
V(S)=\set{x\in \md}{\forall_{\delta \in S }\ \ \delta(x)=0}
\]
and if $A \subseteq \md$ we say that $A$ is a \emph{free variety in $\md$} if $A=V(S)$ for some $S\subseteq \pd$. If $D$ is a free domain and $A$ is a variety in $\md$, we say that $V=A \cap D$ is a \emph{variety in $D$}.
For a point $x\in \md$ we define an ideal $I_x$ in $\pd$ by
\[
I_x = \set{\delta \in \pd}{\delta(x)=0}
\]
and define a variety in $\md$ by
\[
V_x = V(I_x).
\]
Notice that if $A$ is a variety in $\md$ and $x\in A$, then $V_x \subseteq A$.
\begin{defin}\label{free.def.10}
A \emph{Pick pair} is an ordered pair $(D,E)$ such that $D$ is a basic set in $\md$ and $E$ is a subset of $D$ satisfying the following two properties.
\be\label{free.60}
E \text{ is closed with respect to direct sums.}
\ee
\be\label{free.70}
x\in E \implies V_x\cap D \subseteq E.
\ee
\end{defin}
\begin{exam}\label{free.exam.10}
If $D$ is a basic set in $\md$ and $V$ is a variety in $D$, then $(D,V)$ is a Pick pair.
\end{exam}
\begin{proof}
Let $V=A \cap D$ where $A$ is a variety in $\md$. As $D$ and $A$ are closed with respect to direct sums, $V$ is closed with respect to direct sums, i.e., \eqref{free.60} holds. Furthermore, since $V_x \subseteq A$ whenever $x\in A$,
$V_x \cap D \subseteq A \cap D =V$ whenever $x\in A$, i.e., \eqref{free.70} holds.
\end{proof}
\begin{defin}\label{free.def.20}
If $(D,E)$ is a Pick pair, then \emph{Pick data on $E$} is a function $f:E \to \mone$ satisfying the following properties:
\be\label{free.80}
\forall_{x\in E}\ \ f(x) \in \pd(x).
\footnote{Here, $\pd(x) =\set{\pi(x)}{\pi \in \pd}$},
\ee
and $f$ preserves direct sums, i.e.,
\be\label{free.90}
\forall_m\ \ \ x_1,\ldots,x_m \in E \implies f(\oplus_{i=1}^m x_i) = \oplus_{i=1}^m f(x_i).
\ee

\end{defin}
\begin{defin}\label{free.def.30}
If $(D,E)$ is a Pick pair, we say that $(D,E)$ is \emph{norm preserving} if whenever $f$ is bounded Pick data on $E$, there exists a free holomorphic function $F$ on $D$ that extends $f$ and such that
\be\label{free.100}
\sup_{x\in D} \norm{F(x)} = \sup_{x\in E} \norm{f(x)}.
\ee
\end{defin}
It turns out that Pick pairs are  always norm preserving.
\begin{thm}\label{free.thm.20}
(\cite{agmc13b}, Theorem 1.5) If $(D,E)$ is a Pick pair, then $(D,E)$ is norm preserving.
\end{thm}
\begin{proof}
Assume that $(D,E)$ is a Pick pair and $f$ is bounded Pick data on $E$. We wish to show that there exists a bounded free holomorphic function $F$ on $D$ that extends $f$ and such that \eqref{free.100} holds. Fix a sequence $\{\lambda_j\}_{j=1}^\infty$ in $E$ that is dense in $E$ in the coproduct topology and for each $n\ge 1$ let
\[
\Lambda_n = \oplus_{j=1}^n \lambda_j
\]
By $\eqref{free.60}$ $\Lambda_n \in E$. Hence, \eqref{free.80} implies that there exists $\pi_n \in \pd$ such that
\be{}\label{free.110}
f(\Lambda_n)=\pi_n(\Lambda_n).
\ee

Now, fix $x\in V_{\Lambda_n}$. As $\Lambda_n \in E$, \eqref{free.70} implies that $x\in E$. Hence, by \eqref{free.60}, $\Lambda_n \oplus x \in E$. But then, \eqref{free.80} implies that there exists $\rho\in \pd$ such that
\[
f(\Lambda_n \oplus x) = \rho(\Lambda_n \oplus x),
\]
or equivalently, via \eqref{free.90},
\[
f(\Lambda_n)=\rho(\Lambda_n)\qquad \text{ and }\qquad f(x)=\rho(x).
\]
The first equation above together with \eqref{free.110} imply that $\rho(\Lambda_n)=\pi_n(\Lambda_n)$. Therefore, as $x \in V_{\Lambda_n}$, $\rho(x)=\pi_n(x)$. But then the second equation implies that $f(x)=\pi_n(x)$.

Summarizing, in the previous paragraph we showed that if $x\in V_{\Lambda_n}$, then $f(x)=\pi_n(x)$. Hence,
\[
\sup_{x \in V_{\Lambda_n}} \norm{\pi_n(x)} = \sup_{x \in V_{\Lambda_n}} \norm{f(x)} \le \sup_{x \in E} \norm{f(x)}.
\]
Consequently, by Theorem 1.3 in \cite{agmc13b}, there exists a bounded free holomorphic function $F_n$ on $D$ that satisfies
\[
F_n(\Lambda_n) = f(\Lambda_n)\qquad \text{ and }\qquad \sup_{x\in D} \norm{F_n(x)} \le \sup_{x \in E} \norm{f(x)}.
\]
The desired function $F$ can now be obtained by invoking a Montel Theorem as in the proof of Theorem 1.5 in \cite{agmc13b}.
\end{proof}
\subsection{$p$ Norms}
In order to relate the free setting to the classical theory of functions in several complex variables of particular value is the algebraic set $\mdcom$  defined in $\md$ by
\[
\mdcom=V(S)
\]
where
\[
S=\set{x_ix_j-x_jx_i}{i,j=1,\ldots,d}.
\]
We let $\pdcom$ denote the algebra of polynomials in $d$ variables. 
Naturally, $\pdcom$ can be identified with the quotient $\pd / I$ where $I$ is the ideal generated by the set $S$ above. Accordingly, each $p \in \pdcom$ corresponds to a coset $\delta +I$ and conversely to each coset $\delta +I$ corresponds an element of $\pdcom$. When $p$ and $\delta$ are so related we write $p=[\delta]$.

For $p$ an $I \times J$ matrix with entries in $\pdcom$ we define a set $G_p$ in $\c^d$ by the formula
\[
G_p = \set{\lambda \in \c^d}{\norm{p(\lambda)}<1},
\]
For example, if
\[
p(\lambda)=\begin{bmatrix}\lambda_1&&&\\
&\lambda_2&&\\
&&\ddots&\\
&&&\lambda_d
\end{bmatrix},
\]
then $G_p$ equals the polydisc $\d^d$, if
\[
p(\lambda)=\begin{bmatrix}\lambda_1\\
\lambda_2\\
\ddots\\
\lambda_d
\end{bmatrix},
\]
then $G_p$ is the unit ball in $\c^d$ centered at the origin, and if
\[
p(\lambda)=\begin{bmatrix}p_1(\lambda)&&&\\
&p_2(\lambda)&&\\
&&\ddots&\\
&&&p_n(\lambda)
\end{bmatrix}
\]
where $p_1,\ldots,p_n$ are polynomials in $d$ variables,
then $G_p$ is the general polynomial polyhedron in $\c^d$. We say a domain $D$ in $\c^d$ is an \emph{operhedron} if $D=G_p$ for some matrix $p$ with entries in $\pdcom$. (So an operhedron 
is a polynomial polyhedron if and only if it can be written as $G_p$ where $p$ is a diagonal matrix.)

Basic sets in $\md$ and operhedrons in $\c^d$ are related in a natural way described in the following lemma.
\begin{lem}\label{}
If $\delta$ is an $I\times J$ matrix of free polynomials and $p=[\delta]$, then
\[
([\lambda_1],\ldots,[\lambda_d]) \in B_\delta \cap \md_1\qquad  \iff\qquad  (\lambda_1,\ldots,\lambda_d) \in G_p.
\]
\end{lem}
\begin{proof}
Notice first that
\[
([\lambda_1],\ldots,[\lambda_d]) \in \md_1\qquad  \iff\qquad  (\lambda_1,\ldots,\lambda_d) \in \c^d.
\]
Also, as $p=[\delta]$,
\[
\forall_{\lambda \in \c^d}\ \ \delta([\lambda_1],\ldots,[\lambda_d])=p(\lambda_1,\ldots,\lambda_d).
\]
Therefore,
\begin{align*}
([\lambda_1],\ldots,[\lambda_d]) \in B_\delta \cap \md_1 &\iff \norm{\delta([\lambda_1],\ldots,[\lambda_d])} <1\\
&\iff \norm{p(\lambda_1,\ldots,\lambda_d)} <1\\
&\iff (\lambda_1,\ldots,\lambda_d) \in G_p.
\end{align*}
\end{proof}
For $G_p$ an operhedron, we define
\[
\f_p =\set{x \in \mdcom\ }{\ \norm{p(x)} <1},
\]
and for $f\in \hol(G_p)$, define $\norm{f}_p$ by
\be
\label{eq913}
\norm{f}_p = \sup_{x\in \f_p} \norm{f(x)}.
\ee
By $f(x)$ on the right-hand side of \eqref{eq913} we mean we use functional calculus to 
evaluate the holomorphic function $f$ on the $d$-tuple $x$ whose spectrum lies in the domain of $f$.

We let $\hinf_p$ denote the Banach algebra
\[
\hinf_p =\set{f\in \hol(G_p)}{\norm{f}_p<\infty}.
\]
The idea of studying function theory on polydiscs using p-norms dates back to \cite{ag90}. The idea was first considered for general p-operhedrons by Ambrozie and Timotin in the beautiful paper \cite{amti03}. In \cite{babo04} Ball and Bolotnikov considerably extended the work in \cite{amti03} introducing many powerful new tools.

A fundamental connection between p-norms and free analysis is revealed by the following result of two of the authors..
\begin{thm}\label{free.thm.40}(\cite{agmc15b}, Theorem 8.5)
Assume that $\delta$ is an $I\times J$ matrix of free polynomials and let $p=[\delta]$. If $f \in \hinf_p$, then there exists a free holomorphic function $F$ on $B_\delta$ such that
\[
\forall_{\lambda \in G_p}\ \ F([\lambda_1],\ldots,[\lambda_d]) = [f(\lambda_1,\ldots,\lambda_d)]\ \qquad \text{ and }\qquad \sup_{x\in B_\delta} \norm{F(x)}=\norm{f}_p.
\]
Conversely, if $F$ is a bounded free holomorphic function on $B_\delta$ and $f \in \hol(G_\delta)$ is defined by
\[
[f(\lambda_1,\ldots,\lambda_d)]=F([\lambda_1],\ldots,[\lambda_d]),\qquad \lambda \in G_\delta,
\]
then $f\in \hinf_p$ and
\[
\norm{f}_p \le \sup_{x\in B_\delta} \norm{F(x)}.
\]
\end{thm}
The proof of Theorem \ref{free.thm.40} in \cite{agmc15b} required a great deal of heavy lifting. Here, we present a simple proof using the fact that Pick pairs are norm preserving.
\begin{proof}
Let $(D,E) = (B_\delta,B_\delta \cap \mdcom)$. Then $D$ is a basic set, $E$ is a subset of $D$, and as both $B_\delta$ and $\mdcom$ are closed with respect to direct sums, so also $E$ is closed with respect to direct sums, i.e., \eqref{free.60} holds. Also, since $\mdcom$ is an algebraic set in $\md$, it is clear that $V_x \subseteq \mdcom$ whenever $x \in \mdcom$, so that \eqref{free.70} holds. Therefore, $(D,E)$ is a Pick pair.

Now fix $f\in \hinf_p$. As $f$ is analytic on $G_p$ and $\sigma(x) \subseteq G_p$ whenever $x\in E$ 
by Lemma~\ref{lem914},
 we may define a function $g:E \to \mone$ by the formula
\be
\label{eq914}
\notag
g(x)=f(x),\qquad x \in E.
\ee

As $g$ is Pick data on $E$, and
\[
\sup_{x\in E}\norm {g(x)} = \sup_{x\in \f_p} \norm{f(x)} =\norm{f}_p,
\]
it follows from Theorem \ref{free.thm.20} that there exists $F$ with the desired properties.

To prove the converse, note that since $F$ is locally approximable by a sequence in $\pd$, we get that
$f$ is locally approximable by a sequence in $\pdcom$, and hence is holomorphic.
The inequality $\norm{f}_p \le \sup_{x\in B_\delta} \norm{F(x)}$ is automatic.
\end{proof}
\begin{lem}
\label{lem914}
Let $x \in B_\delta \cap \mdcom$, and let $p=[\delta]$.
Then $\sigma(x) \subseteq G_p$.
\end{lem}
\begin{proof}
Let $x \in \mn^d$.
Suppose $\l \in \sigma(x)$. Let $\xi$ be a unit vector in $\c^J$ such that
$\| p (\l) \xi \| = \| p( \l ) \|$.

Choose an orthonormal basis in $\c^n$ with respect to which $x$ is upper triangular.
Then for some $1 \leq j \leq n$,  we have $\l = (\l_1, \dots, \l_d)$ forms the $j^{\rm th}$ diagonal
entry of $x = (x_1, \dots, x_d)$  (see \cite{cur88} or \cite{coh14}). Write $v$ for this particular $e_j$.
So for any polynomial $q$ we have $q(x) v = q(\l) v + w$, where $w \perp v$.
So \[
p(x) (\xi \otimes v) = p(\l) (\xi ) \otimes v + \eta,
\] where $\eta \perp \c^J \otimes v$.
Therefore $\| p(x) \| \geq \| p(\l) \|$, so $\l \in G_p$.
\end{proof}

\subsection{$p$ $V$ Norms}
\label{ssec95}
It is easy to adapt the results of the previous subsection to the problem of extending holomorphic functions defined on varieties $V$ in $G_p$. The key is to replace the set $\f_p$ with the set $\f_{p,V}$, consisting of all $x\in \f_p$ with spectrum in $V$ and such that $F(x)$ depends only on the values of $F$ on $V$ whenever $F$ is holomorphic on a neighborhood of $V$.

The set $\f_{p,V}$ can be defined concretely as follows. For $\lambda \in G_p$ define $X_\lambda$ by
\[
X_\lambda = \set{x\in \mdcom}{\sigma(x)=\{\lambda\}}.
\]
For $V\subseteq G_p$ define $X_V$ to consist of the set of all finite direct sums
\[
x=\bigoplus_{i=1}^n x_{i},
\]
where $x_i \in X_{\lambda_i}$ for each $i$, and $\lambda_1,\ldots,\lambda_n$ are points in $V$. Finally, let $Y_V$ be defined by
\[
Y_V = \set{y\in \mdcom}{\exists_{x\in X_V}\ \  y \text{ is similar to } x}.
\]

Evidently, $Y_V$ is the collection of all pairwise commuting $d$-tuples of matrices with spectrum in $V$.
(This can be seen by decomposing the space on which the $d$-tuple acts into generalized eigenspaces.)
 We define a functional calculus for elements of $Y_V$ as follows. If $f$ is holomorphic on a neighborhood of $\lambda$ and $x\in X_\lambda$, then $f(x)$ can be defined by plugging $x$ into the power series expansion of $f$ at $\lambda$ (which will result in a finite sum). More generally, if $f$ is holomorphic on $V$, and
\[
y = S^{-1}\ \bigoplus_{i=1}^n x_{i}\ S\  \in Y_V,
\]
where $x_i \in X_{\lambda_i}$ for each $i$,
we define $f(y)$ by
\be
\label{eq916}
f(y)=S^{-1}\ \bigoplus_{i=1}^n f(x_{i})\ S
\ee
We shall say that $y \in Y_V$ is {\em subordinate} to $V$ if, whenever 
$f$ is holomorphic on a on a neighborhood $U$ of $\sigma(Y)$ in $\c^d$, and
vanishes
on $U \cap V$, 
 then $f(y) = 0$.
 
The promised set $\f_{p,V}$ is defined by
\[
\f_{p,V}=\set{y\in Y_V}{y \text{ is subordinate to } V}
\cap \f_p .
\]
As before, for $f\in \hol(V)$, define $\norm{f}_{p,V}$ by
\[
\norm{f}_{p,V} = \sup_{y\in \f_{p,V}} \norm{f(y)}.
\]
and let $\hinf_{p,V}$ denote the Banach algebra
\[
\hinf_{p,V} =\set{f\in \hol(V)}{\norm{f}_{p,v}<\infty}.
\]
\begin{exam}
Let $d =2$, let 
\[
p(\l) \ = \
\begin{bmatrix}
\l_1 & 0 \\
0 & \l_2
\end{bmatrix}, 
\]
so $G_p = \d^2$, and let 
\[
V \ = \ \{ \l \in \d^2 : \l_1^2 = \l_2^2 \} .
\]
Then $Y_V$ is the set of pairs $y = (y_1,y_2)$
 of commuting matrices whose spectrum is in $\d^2$,
and $\f_{p,V}$ is the subset for which both $y_1$ and $y_2$ have norm less than $1$ and in addition satisfy $y_1^2 = y_2^2$.
This is slightly larger than just the diagonalizable pairs of strict contractions; for example it contains the pair
\[
\begin{bmatrix}
c & d & 0 & 0 \\
0 & c & 0 & 0 \\
0 & 0 & c & d \\
0 & 0 & 0 & c
\end{bmatrix}
\text{ and }
\begin{bmatrix}
-c & -d  & 0 & 0\\
0 & -c & 0 & 0 \\
0 & 0 & c & d \\
0 & 0 & 0 & c
\end{bmatrix}
\]
whenever $c$ and $d$ are small enough to give contractions.
But $\f_{p,V}$ does not contain the pair
\[
\begin{bmatrix}
c & d \\
0 & c
\end{bmatrix}
\text{ and }
\begin{bmatrix}
-c & d \\
0 & -c
\end{bmatrix}
\]
whenever $c$ and $d$ are non-zero, even though this pair is in $Y_V$.
\end{exam}
A very similar proof to the one  for Theorem \ref{free.thm.40} yields the following.
\begin{thm}\label{free.thm.50}
Assume that $\delta$ is an $I\times J$ matrix of free polynomials and let $p=[\delta]$. Let $A$ be an algebraic set\footnote{i.e. $A$ is the common set of 0's of a collection of polynomials in $\pdcom$}
 in $\c^d$ and let $V=A \cap G_p$.
If $f \in \hinf_{p,V}$, then there exists a free holomorphic function $F$ on $B_\delta$ such that
\[
\forall_{\lambda \in V}\ \ F([\lambda_1],\ldots,[\lambda_d]) = [f(\lambda_1,\ldots,\lambda_d)]\ \qquad \text{ and }\qquad \sup_{x\in B_\delta} \norm{F(x)}=\norm{f}_{p,V}.
\]
Conversely, if $F$ is a bounded free holomorphic function on $B_\delta$ and $f \in \hol(V)$ is defined by
\[
[f(\lambda_1,\ldots,\lambda_d)]=F([\lambda_1],\ldots,[\lambda_d]),\qquad \lambda \in V,
\]
then $f\in \hinf_{p,V}$ and
\[
\norm{f}_p \le \sup_{x\in B_\delta} \norm{F(x)}.
\]
\end{thm}
\begin{proof}
Let $(D,E) = (B_\delta, \f_{p,V})$. Then  \eqref{free.60} holds, and since
$A$ is an algebraic set, so does \eqref{free.70}. Therefore, $(D,E)$ is a Pick pair.

Now fix $f\in \hinf_{p,V}$ and $y \in E$.
 We may define a function $g:E \to \mone$ by  formula \eqref{eq916}.
Notice that this is well-defined. If at some $x_i$ with spectrum $\{ \l \} $ we choose two 
different holomorphic functions $g,h$ which agree with $f$ on a neighborhood of $x_i$ in
$V$, then $g(x_i) = h(x_i)$ since $x_i$ is subordinate to $V$.

As $g$ is Pick data on $E$, and
\[
\sup_{y\in E}\norm {g(x)} = \sup_{y\in \f_{p,V}} \norm{f(x)} =\norm{f}_{p,V},
\]
it follows from Theorem \ref{free.thm.20} that there exists $F$ with the desired properties.
\end{proof}

\subsection{A Sharp Commutative Cartan Extension Theorem}

We can now use the non-commutative result Theorem \ref{free.thm.50} to
prove a commutative extension theorem with sharp bounds for algebraic sets in operhedrons.

\begin{thm}\label{free.thm.60}
Let $p\in \pdcomij$ and assume that $V$ is an
algebraic set
 in $G_p$.
 If $f\in \hinf_{p,V}$, then there exists an extension  $F\in \hinf_p$ such that
\be
\label{eq920}
F|V =f\qquad \text{}\qquad \norm{F}_p = \norm{f}_{p,V}.
\ee
 Moreover, the norm estimate in \eqref{eq920} cannot be improved.
  If $F\in \hinf_p$, then $F|V\in \hinf_{p,V}$ and $\norm{F|V}_{p,V}\le \norm{F}_p$.
\end{thm}
\begin{proof}
Let $f\in \hinf_{p,V}$. Choose $\delta$ in $\pdij$ so that
$p = [ \delta]$. By Theorem~\ref{free.thm.50} there exists a free holomorphic function 
$\Phi$ on $B_\delta$ with norm equal to $\norm{f}_{p,V}$ that extends $f$.
Define $F = \Phi |_{G_p}$. Then 
\be
\| F \|_p \ \leq \ \| \Phi \|_{B_\delta} \ = \ \norm{f}_{p,V} .
\label{eq921}
\ee
But if $F|V = f$, we must have
\[
\| F \|_p \ \geq \ \norm{f}_{p,V} ,
\]
since we are taking the supremum over a larger set.
Therefore we must have equality in \eqref{eq921}.
\end{proof}

%\green
%I think extending Theorem \ref{free.thm.60} to the case where $V$ is an analytic subvariety 
%of $G_p$ is straightforward - we just redo Section \ref{ssec95} with analytic 
%instead of algebraic varieties. I don't think it requires anything more than cosmetic changes.
%
%More interesting is to allow $G_p$ to become any pseudoconvex domain.
%Is there a way of approximating pseudoconvex domains by operhedra so that one can deduce something close to 
%Cartan's original theorem? Or, can one use the hereditary calculus to apply plurisubhamronic functions to operators, so one can find an operator domain corresponding to a plurisubharmonic exhaustion?
%\black

\bibliography{references}
\end{document}